\documentclass[12pt]{amsart}
\usepackage[T1]{fontenc}
\usepackage[latin9]{inputenc}
\usepackage{geometry}
\usepackage{enumitem}
\usepackage{amstext}
\usepackage{amssymb}
\usepackage{stmaryrd}
\usepackage[unicode=true,pdfusetitle,
 bookmarks=true,bookmarksnumbered=false,bookmarksopen=false,
 breaklinks=false,pdfborder={0 0 1},backref=false,colorlinks=false]
 {hyperref}

\makeatletter
\usepackage{amsthm}
\newcommand\thmsname{Theorem}
 \newcommand\nm@thmtype{thm}
 \theoremstyle{plain}
 
 \newenvironment{namedthm}[1]{
   \renewcommand\thmsname{#1}\renewcommand\nm@thmtype{namedtheorem}
   \begin{\nm@thmtype}
}
   {\end{\nm@thmtype}
}

\usepackage{rotating}

\usepackage{amsfonts}\usepackage{amscd}\usepackage[all]{xy}\usepackage{color}\usepackage{tikz}\usepackage{amsthm}

\numberwithin{equation}{section}

\newtheorem{thm}{Theorem}[section]
\newtheorem*{thm*}{Theorem}
\newtheorem{cor}[thm]{Corollary}
\newtheorem*{cor*}{Corollary}
\newtheorem{lem}[thm]{Lemma}
\newtheorem*{lem*}{Lemma}
\newtheorem{prop}[thm]{Proposition}
\newtheorem*{prop*}{Proposition}

\newtheorem*{conjecture*}{Conjecture}

\newtheorem*{fact*}{Conjecture}

\newtheorem*{criterion*}{Criterion}

\newtheorem*{algorithm*}{Algorithm}

\newtheorem*{ax*}{Axiom}

\newtheorem*{assumption*}{Assumption}

\newtheorem*{question*}{Question}

\theoremstyle{remark}

\newtheorem{rem}[thm]{Remark}
\newtheorem*{rem*}{Remark}

\newtheorem*{rems*}{Remarks}

\newtheorem*{claim*}{Claim}

\newtheorem*{exercise*}{Exercise}

\newtheorem*{note*}{Note}
\newtheorem{notation}[thm]{Notation}
\newtheorem*{notation*}{Notation}

\newtheorem*{summary*}{Summary}

\newtheorem*{acknowledgement*}{Acknowledgement}

\newtheorem*{conclusion*}{Conclusion}

\theoremstyle{definition}

\newtheorem{defn}[thm]{Definition}
\newtheorem*{defn*}{Definition}
\newtheorem{example}[thm]{Example}
\newtheorem*{example*}{Example}

\newtheorem*{examples*}{Examples}

\newtheorem*{problem*}{Problem}

\newtheorem*{xca*}{Exercise}

\newtheorem*{xcas*}{Exercises}

\newtheorem*{condition*}{Condition}

\newtheorem{proviso}[thm]{Proviso}

\author[Rolin]{Jean-Philippe Rolin}
\address{Universit\'e de Bourgogne Franche-Comt\'e, IMB,
CNRS UMR 5584,
9 Avenue Savary
BP47870,
F-21078 Dijon Cedex, France}
\email{jean-philippe.rolin@u-bourgogne.fr}
\urladdr{http://rolin.perso.math.cnrs.fr/}

\author[Servi]{Tamara Servi}
\address{Institut de Math\'ematiques de Jussieu -- Paris Rive Gauche \\
	Universit\'{e} Paris Cit\'{e} and Sorbonne Universit\'{e}, CNRS, IMJ-PRG, F-75013 Paris, France}
\email{tamara.servi@imj-prg.fr}
\urladdr{https://tamaraservi.github.io/}

\author[Speissegger]{Patrick Speissegger}
\address {Department of Mathematics and Statistics, McMaster University, 1280
Main Street West, Hamilton, Ontario L8S 4K1, Canada}
\email {speisseg@math.mcmaster.ca}

\AtBeginDocument{
  
}

\makeatother

\begin{document}
\title{Transasymptotic expansions of o-minimal germs}
\begin{abstract}
Given an o-minimal expansion $\mathbb{R}_{\mathcal{A}}$ of the real
ordered field, generated by a generalized quasianalytic class $\mathcal{A}$,
we construct an explicit truncation closed ordered differential field
embedding of the Hardy field of the expansion $\mathbb{R}_{\mathcal{A},\exp}$
of $\mathbb{R}_{\mathcal{A}}$ by the unrestricted exponential function,
into the field $\mathbb{T}$ of transseries. We use this to prove
some non-definability results. In particular, we show that the restriction
to the positive half-line of Euler's Gamma function is not definable
in the structure $\mathbb{R}_{\text{an}^{*},\exp}$, generated by
all convergent generalized power series and the exponential function,
thus establishing the non-interdefinability of the restrictions to
a neighbourhood of $+\infty$ of Euler's Gamma and of the Riemann
Zeta function. 
\end{abstract}

\maketitle

\section{Introduction }

Given a generalized quasianalytic class $\mathcal{A}$ (see Definition
\ref{def: QA class-1}), we consider the expansion $\mathbb{R}_{\mathcal{A}}$
of the real ordered field generated by the functions in $\mathcal{A}$.
The first two authors proved in \cite{rolin_servi:qeqa} that $\mathbb{R}_{\mathcal{A}}$
is o-minimal, and indeed all examples of polynomially bounded o-minimal
expansions of the real field with smooth cell decomposition known
so far can be presented as reducts of some $\mathbb{R}_{\mathcal{A}}$,
for a suitable $\mathcal{A}$. If $\exp\restriction\left[0,1\right]\in\mathcal{A}$,
a result of \cite{vdd:speiss:multisum} shows that the expansion $\mathbb{R}_{\mathcal{A},\exp}$
of $\mathbb{R}_{\mathcal{A}}$ by the unrestricted exponential function
is also o-minimal. It is well known that the germs at $+\infty$ of
the unary functions definable in an o-minimal expansion $\mathfrak{R}$
of the reals form a Hardy field, which can be made into an elementary
extension of $\mathfrak{R}$. 

In the Main Theorem (Section \ref{sec:The-Main-Theorem} below), we
construct an explicit truncation closed ordered differential field
embedding of the Hardy field $\mathcal{H}\left(\mathbb{R}_{\mathcal{A},\exp}\right)$
of $\mathbb{R}_{\mathcal{A},\exp}$ into the field $\mathbb{T}$ of
logarithmic-exponential series, or \emph{transseries} (defined and
studied in \cite{vdd_macintyre_marker:log-exp_series,adh-bigbook}).
This embedding provides a \emph{transasymptotic expansion} (see Definition
\ref{def: trans expansion}) of a definable germ, in the scale of
transmonomials. The image $\mathbb{T}_{\mathcal{A}}$ of $\mathcal{H}\left(\mathbb{R}_{\mathcal{A},\exp}\right)$
under this embedding is a subfield of $\mathbb{T}$. Understanding
what characterizes the properties of the elements of $\mathbb{T}_{\mathcal{A}}$
-- such as convergence, summability, nature of support or coefficients
-- then helps us establish necessary conditions for a real function
to be definable in $\mathbb{R}_{\mathcal{A},\exp}$ (see Section \ref{sec:Applications-to-non-definability-1}
for examples).

This work was initially motivated by the following question: consider
the restrictions to suitable real neighbourhoods of $+\infty$ of
the Riemann Zeta function $\zeta$ and of Euler's gamma function $\Gamma$.
It is known that neither of these functions is definable in the expansion
$\mathbb{R}_{\text{an},\exp}$ of the real field by restricted analytic
functions and the exponential function \cite{dmm:series}. However,
van den Dries and Speissegger proved that these two functions are
definable, respectively, in the o-minimal structure $\mathbb{R}_{\text{an}^{*},\exp}$,
generated by all convergent generalized power series and the exponential
function \cite{vdd:speiss:gen}, and in the o-minimal structure $\mathbb{R}_{\mathcal{G},\exp}$,
generated by series that are multisummable in the positive real direction
and the exponential function \cite{vdd:speiss:multisum}. They also
proved that the restriction of $\zeta$ to the half-line $\left(1,+\infty\right)$
is not definable in $\mathbb{R}_{\mathcal{G},\exp}$ \cite{vdd:speiss:multisum}.
The following two questions have remained unanswered until very recently:
\begin{enumerate}
\item Is the restriction of $\Gamma$ to the positive half-line definable
in $\mathbb{R}_{\text{an}^{*},\exp}$?
\item If not, is there an o-minimal structure in which both $\zeta\restriction\left(1,+\infty\right)$
and $\Gamma\restriction\left(0,+\infty\right)$ are definable?
\end{enumerate}
Here, we apply our Main Theorem to answer the first question in the
negative (Corollary \ref{cor: applications}(1)), and in \cite{rss:multisummability_generalized}
we answer the second question. Incidentally, we also give here a more
direct proof of the fact that $\zeta\restriction\left(1,+\infty\right)$
is not definable in $\mathbb{R}_{\mathcal{G},\exp}$ (nor in any structure
$\mathbb{R}_{\mathcal{A},\exp}$ such that $\mathcal{A}$ is a \emph{classical
}quasianalytic class, see Definition \ref{def: tr. closed, natural, classical-1}
and Proposition \ref{prop: zeta not def}).

\medskip{}

The Main Theorem relies on the more general Embedding Theorem \ref{thm: embedding}.
The latter is about embedding a certain type of ordered field $\mathcal{F}$
of real germs (not necessarily a Hardy field, and not necessarily
a field of germs definable in an o-minimal structure) into a suitable
ordered Hahn field $\mathbb{R}\left(\left(G\right)\right)$ (not necessarily
a transserial field and not necessarily a field endowed with $\exp$
and $\log$), see Proviso \ref{proviso} and Definition \ref{def: F}.
It is an ordered field embedding $\phi$ which provides a transasymptotic
expansion for the germs in $\mathcal{F}$ in the scale of monomials
in $\phi^{-1}\left(G\right)$.

The field $\mathcal{F}$ is constructed from a group $\mathcal{M}$
of monomial germs (Definition \ref{def: monmial germs}) essentially
by considering the germs of a generalized quasianalytic class $\mathcal{A}$
in restriction to \textquotedblleft monomial curves\textquotedblright .
The proof of the Embedding Theorem \ref{thm: embedding} relies on
two results: Corollary \ref{cor: monom of F} explains how to monomialize
the elements of $\mathcal{F}$ using techniques from local resolution
of singularities; the Splitting Lemma \ref{lem: splitting} explains
how to suitably decompose the elements of $\mathcal{F}$ into the
sum of two germs using the properties of $\mathcal{A}$, and this
is used to prove that the embedding is truncation closed.

\medskip{}

The proof of the Main Theorem uses a general quantifier elimination
result for structures of type $\mathbb{R}_{\mathcal{A}}$ proved in
\cite{rolin_servi:qeqa}. This result is what allows us to describe
the elements of $\mathcal{H}\left(\mathbb{R}_{\mathcal{A}}\right)$
first, and then, using \cite[Theorem B]{vdd:speiss:multisum}, the
elements of $\mathcal{H}\left(\mathbb{R}_{\mathcal{A},\exp}\right)$,
as \emph{terms} of a suitable language, involving symbols for the
functions in $\mathcal{A}$ and symbols for $\exp$ and $\log$ (Theorem
\ref{thm: strong QE-1}). There is a natural way to associate a transseries
to each such term, yielding a direct connection between germs and
transseries: given a germ $f\in\mathcal{H}\left(\mathbb{R}_{\mathcal{A},\exp}\right)$,
we may choose a term $t_{f}$ representing it and consider the transseries
naturally associated to the term $t_{f}$. There are, however, multiple
choices possible for the term $t_{f}$, and it is not clear a priori
that any of these choices would provide a well defined order-preserving
map from $\mathcal{H}\left(\mathbb{R}_{\mathcal{A},\exp}\right)$
to $\mathbb{T}$ (see Example \ref{exa: tilde and hat-1}). We show
that there is a suitable choice of $t_{f}$, for all $f\in\mathcal{H}\left(\mathbb{R}_{\mathcal{A},\exp}\right)$,
leading to the desired embedding.

\medskip{}

The paper is organized as follows. In Section \ref{sec: general objects and results}
we define the notion of transasymptotic expansion for real germs.
In Section \ref{sec:Generalized-quasianalytic-class} we introduce
generalized quasianalytic classes and state and prove the Embedding
Theorem \ref{thm: embedding}. Section \ref{sec:The-Main-Theorem}
is devoted to the statement and proof of the Main Theorem. In Section
\ref{sec:Applications-to-non-definability-1} we give some applications
of the Main Theorem.

\section{\label{sec: general objects and results}Transasymptotic expansions
of real germs}
\begin{notation*}
Throughout this paper, lower-case letters $x,x_{i},y,\ldots$ will
denote \textquotedblleft geometric\textquotedblright{} variables,
ranging in subsets of $\mathbb{R}^{n}$ for some $n>0$, and upper-case
letters $X,X_{i},Y,\ldots$ will denote \textquotedblleft formal\textquotedblright{}
variables, appearing in formal series and transseries.

The geometric variables $x,x_{i},\ldots$ will usually range over
a neighbourhood of $0\in\mathbb{R}^{n}$, whereas the geometric variable
$y$ will range over some neighbourhood of $+\infty$ in $\mathbb{R}$.
Accordingly, the formal variables $X,X_{i},\ldots$ will be positive
and infinitely small, whereas $Y$ will be positive and infinitely
large, with respect to $\mathbb{R}$.
\end{notation*}

\subsection{Generalized power series\label{subsec:Generalized-power-series}}

For $X=\left(X_{1},\ldots,X_{\ell}\right)$, we let $\mathbb{R}\left\llbracket X^{*}\right\rrbracket $
be the $\mathbb{R}$-algebra of \emph{generalized power series }of
the form
\begin{equation}
F\left(X\right)=\sum_{\mathbf{r}=\left(r_{1},\ldots,r_{N}\right)\in[0,+\infty)^{\ell}}a_{\mathbf{r}}X_{1}^{r_{1}}\cdots X_{\ell}^{r_{\ell}},\label{eq:gen seris}
\end{equation}
where $a_{\mathbf{r}}\in\mathbb{R}$ and $\text{Supp}\left(F\right):=\left\{ \mathbf{r}:\ a_{\mathbf{r}}\not=0\right\} $
is contained in a cartesian product of $\ell$ well-ordered subsets
of $[0,+\infty$). If $m\leq N$, we also consider the subring $\mathbb{R}\left\llbracket X_{1}^{*},\ldots,X_{m}^{*},X_{m+1},\ldots,X_{\ell}\right\rrbracket $
of $\mathbb{R}\left\llbracket X^{*}\right\rrbracket $ consisting
in all generalized power series whose support is contained in $[0,+\infty)^{m}\times\mathbb{N}^{\ell-m}$.

The basic properties of generalized power series can be found in \cite[Section 4]{vdd:speiss:gen}.

\subsection{Ordered groups with powers\label{subsec:Ordered-groups-with}}
\begin{defn}[Ordered groups with $\mathbb{K}$-powers]
\label{def: ordered group with powers} Let $\left(G;\ 1,\cdot,<\right)$
be an abelian totally ordered group (denoted multiplicatively) and
$\mathbb{K}\subseteq\mathbb{R}$ be a subfield. We say that $G$ is
an \emph{ordered group with $\mathbb{K}$-powers} if $G$ is an ordered
$\mathbb{K}$-vector space (where scalar multiplication is denoted
exponentially):
\begin{itemize}
\item $\forall g\in G,\ g^{0}=1$ and $g^{1}=g$
\item $\forall g_{1},g_{2}\in G,\ \forall r\in\mathbb{K},\ \left(g_{1}\cdot g_{2}\right)^{r}=g_{1}^{r}\cdot g_{2}^{r}$
\item $\forall g_{1},g_{2}\in G,\ \forall r\in\mathbb{K}_{\geq0},\ g_{1}\leq g_{2}\Longrightarrow g_{1}^{r}\leq g_{2}^{r}$
\item $\forall g\in G,\ \forall r,s\in\mathbb{K},\ g^{r+s}=g^{r}\cdot g^{s}$
\item $\forall g\in G,\ \forall r,s\in\mathbb{K},\ \left(g^{r}\right)^{s}=g^{rs}$
\end{itemize}
We denote $G^{<1}=\left\{ g\in G:\ g<1\right\} $ and $G^{>1}=\left\{ g\in G:\ g>1\right\} $.
\end{defn}
\begin{rem}
\label{rem:K-vector space}It follows from the definition that, if
$r\in\mathbb{K}^{*}$, then for all $g\in G,\ \left(g^{r}\right)^{\frac{1}{r}}=\left(g^{\frac{1}{r}}\right)^{r}=g$
and $g^{r}\cdot g^{-r}=1$. Hence, $g\longmapsto g^{r}$ is a group
isomorphism, which is order-preserving if $r>0$ and order-reversing
if $r<0$. In particular, if $r,s\in\mathbb{K}$ and $g\in G^{<1}$,
then $r<s\Longrightarrow g^{r}>g^{s}$.
\end{rem}
\begin{defn}[Hahn series]
\label{def: Hahn series}Given an ordered Abelian group $G$, we
denote by $\mathbb{R}\left(\left(G\right)\right)$ the \emph{Hahn
field} with monomials in $G$ and coefficients in $\mathbb{R},$ whose
elements can be written as
\[
\sigma=\sum_{g\in G}b_{g}g,
\]
where $b_{g}\in\mathbb{R}$ for each $g\in G$, and the support $\text{Supp\ensuremath{\left(\sigma\right)}:=\ensuremath{\left\{  g\in G:\ b_{g}\not=0\right\} } }$
is a reverse well-ordered subset of $G$. Hence, $\sigma$ can also
be written as $\sum_{\gamma<\gamma_{0}}c_{\gamma}g_{\gamma}$, where
$\gamma_{0}$ is an ordinal, $c_{\gamma}\in\mathbb{R}\setminus\left\{ 0\right\} $
and the sequence $\left\{ g_{\gamma}\right\} _{\gamma<\gamma_{0}}\subseteq G$
is strictly decreasing. Therefore $g_{0}=\max\ \text{Supp}\left(\sigma\right)$
and $c_{0}=b_{g_{0}}$. We define the leading monomial, the leading
coefficient and the leading term of $\sigma$, respectively, as
\[
\text{lm}\left(\sigma\right)=g_{0},\ \text{lc}\left(\sigma\right)=c_{0},\ \text{lt}\left(\sigma\right)=c_{0}g_{0}.
\]
\end{defn}
\begin{rem}
\label{rem: action of gen series on groups}Let $G$ be an ordered
group with $\mathbb{K}$-powers and $F\in\mathbb{R}\left\llbracket X_{1}^{*},\ldots,X_{\ell}^{*}\right\rrbracket $
such that $\text{Supp}\left(F\right)\subseteq\left(\mathbb{K}_{\geq0}\right)^{\ell}$.
Then $F$ acts on $\left(G^{<1}\right)^{\ell}$ is the following way:
if $F$ is as in \eqref{eq:gen seris} and $g_{1},\ldots,g_{\ell}\in G^{<1}$,
then
\[
F\left(g_{1},\ldots,g_{\ell}\right)=\sum_{g\in G}b_{g}g\in\mathbb{R}\left(\left(G\right)\right),
\]
where
\[
b_{g}=\left(\sum_{\mathbf{r}\in\text{Supp}\left(F\right):\ \prod_{i=1}^{\ell}g_{i}^{r_{i}}=g}a_{\mathbf{r}}\right)
\]
and the above sum is finite by Remark \ref{rem:K-vector space} and
Neumann's Lemma \cite{neumann:ordered_division_rings}.
\end{rem}

\subsection{Transasymptotic scales of germs\label{subsec:Transasymptotic-scales-of}}
\begin{defn}
\label{def: order on germs}We let $\mathcal{R}$ be the ring of all
germs at $+\infty$ of real functions (with respect to pointwise sum
and product). We introduce two partial orders on $\mathcal{R}$: for
$f,g\in\mathcal{R}$,
\begin{itemize}
\item $f\leq g\Longleftrightarrow$ there are representatives $f_{0},g_{0}$
of $f,g$ respectively on some half-line $\left(a,+\infty\right)\subseteq\mathbb{R}$
such that
\[
\forall x\in\left(a,+\infty\right),\ f_{0}\left(x\right)\leq g_{0}\left(x\right)\ \ \text{\ensuremath{\left(\mathrm{pointwise\ order}\right)}}.
\]
\item $f\preceq g\Longleftrightarrow{\displaystyle \lim_{x\rightarrow+\infty}\frac{f\left(x\right)}{g\left(x\right)}}\in\mathbb{R}$
$\ \ \ $(valuation order).
\end{itemize}
Thus $f\prec g\Longleftrightarrow f=o_{+\infty}\left(g\right)$ and
$f\asymp g\Longleftrightarrow\left(f\preceq g\text{ and }g\preceq f\right)\Longleftrightarrow\exists c\in\mathbb{R}^{*}$
such that $\lim_{x\rightarrow+\infty}\frac{f\left(x\right)}{g\left(x\right)}=c$.

We write $f\sim_{+\infty}g$ if $\lim_{x\rightarrow+\infty}\frac{f\left(x\right)}{g\left(x\right)}=1$.
\end{defn}
\begin{defn}[Monomials]
\label{def: monmial germs}Let $\mathcal{M}\subseteq\mathcal{R}$.
We say that $\mathcal{M}$ is a \emph{group of monomials} if:
\begin{itemize}
\item The germs in $\mathcal{M}$ take values in $\left(0,+\infty\right)$.
\item $\mathcal{M}$ is a totally ordered (with respect to the valuation
order) multiplicative subgroup of $\mathcal{R}$.
\end{itemize}
It follows that $\mathcal{M}$ is also totally ordered with respect
to the pointwise order and the two orders coincide on $\mathcal{M}$:
\[
\forall\mathfrak{m},\mathfrak{n}\in\mathcal{M},\ \mathfrak{m}<\mathfrak{n}\Longleftrightarrow\mathfrak{m}\prec\mathfrak{n}.
\]

\end{defn}
\begin{defn}[Transasymptotic scale]
\label{def: asymptotic scale}Given a group of monomials $\mathcal{M}\subseteq\mathcal{R}$
and an ordinal $\gamma$, we say that a sequence $\left(\mathfrak{m}_{\alpha}\right)_{\alpha<\gamma}\subseteq\mathcal{M}$
is a \emph{transasymptotic scale} if
\[
\forall\alpha<\beta<\gamma,\ \mathfrak{m}_{\beta}\prec\mathfrak{m}_{\alpha}.
\]
\end{defn}
\begin{defn}[Transasymptotic expansion]
\label{def: trans expansion}Let $\mathcal{M}\subseteq\mathcal{R}$
be a group of monomials and $\mathcal{F}\subseteq\mathcal{R}$ be
an ordered (with respect to the pointwise order) subfield of $\mathcal{R}$
containing $\mathcal{M}$. Let $G$ be an ordered group with $\mathbb{K}$-powers
(see Definition \ref{def: ordered group with powers}) and suppose
that $\phi:\mathcal{F}\longrightarrow\mathbb{R}\left(\left(G\right)\right)$
is an ordered field embedding mapping $\mathcal{M}$ into $G$.

We say that $\phi$ is \emph{truncation closed} if for every $\sigma=\sum_{\alpha<\gamma}c_{\alpha}g_{\alpha}\in\text{Im}\left(\phi\right)$
and every ordinal $\beta<\gamma$, there exists a (necessarily unique)
germ $f_{\beta}\in\mathcal{F}$ such that
\[
\phi\left(f_{\beta}\right)=\sum_{\alpha<\beta}c_{\alpha}g_{\alpha}.
\]
In particular, for all $\alpha<\gamma$, $g_{\alpha}\in\text{Im}\left(\phi\right)$
and $\left(\phi^{-1}\left(g_{\alpha}\right)\right)_{\alpha<\gamma}$
is a transasymptotic scale.

Finally, a truncation closed embedding $\phi$ as above provides a
\emph{transasymptotic expansion} for the germs in $\mathcal{F}$ (with
real coefficients and transasymptotic scales in $\mathcal{M}$) if
for every $f\in\mathcal{F}$, writing $\phi\left(f\right)=\sum_{\alpha<\gamma}c_{\alpha}g_{\alpha}$,
we have, for every $\beta<\gamma$, 
\[
f-f_{\beta}\sim_{+\infty}c_{\beta}\phi^{-1}\left(g_{\beta}\right),
\]
where $f_{\beta}$ is the unique germ such that $\phi\left(f_{\beta}\right)=\sum_{\alpha<\beta}c_{\alpha}g_{\alpha}$.
\end{defn}

\section{Generalized Quasianalytic Classes\label{sec:Generalized-quasianalytic-class}}
\begin{defn}[Generalized quasianalytic class (GQC)]
\label{def: QA class-1} A \emph{generalized quasianalytic class
(GQC)} is a collection
\[
\mathcal{A}=\left\{ \mathcal{A}_{m,n,\mathbf{r}}:\ m,n\in\mathbb{N},\ \mathbf{r}=\left(r_{1},\ldots,r_{m+n}\right)\in(0,+\infty)^{m+n}\right\} 
\]
of $\mathbb{R}$-algebras of functions 
\[
f:[0,r_{1})\times\cdots\times[0,r_{m})\times\left(-r_{m+1},r_{m+1}\right)\times\cdots\times\left(-r_{m+n},r_{m+n}\right)\longrightarrow\mathbb{R}
\]
which are continuous on their domain and $C^{1}$ on the interior
of their domain, and which satisfy the properties listed in \cite[Proviso 1.20]{rolin_servi:qeqa}.
In particular, if $\mathcal{A}_{m,n}$ denotes the collection of germs
at the origin of $\mathcal{A}_{m,n,\mathbf{r}}$ (for $\mathbf{r}$
in $\left(0,+\infty\right)^{m+n}$), then there is an \textbf{injective}
$\mathbb{R}$-algebra morphism
\begin{equation}
\widehat{}:\mathcal{A}_{m,n}\longrightarrow\mathbb{R}\left\llbracket X_{1}^{*},\ldots,X_{m}^{*},X_{m+1},\ldots,X_{m+n}\right\rrbracket \label{eq:qa morphism hat-1}
\end{equation}
which is compatible with certain operations involved in resolution
of singularities (see \cite[1.15]{rolin_servi:qeqa}), such as ramifications,
blow-ups, monomial division and a restricted form of composition.
\end{defn}
\begin{rem}
\label{rem: gen variables}In what follows there is no harm in assuming
that all variables are generalized, i.e. $n=0$ and $X^{*}=\left(X_{1},\ldots,X_{m}\right)$.
Hence, by an abuse of notation, we denote by $\mathcal{A}_{m}$ the
collection of germs $\mathcal{A}_{m,0}$. We also write $f\in\mathcal{A}$
as a shorthand for $f\in\mathcal{A}_{m},$ for some $m\in\mathbb{N}^{*}$. 
\end{rem}
\begin{defn}
We denote by $\widehat{\mathcal{A}_{m}}\subseteq\mathbb{R}\left\llbracket X^{*}\right\rrbracket $
the image of the quasianalyticity morphism $\ \widehat{}\ $ in \ref{eq:qa morphism hat-1}.
We let $\mathbb{K}\subseteq\mathbb{R}$ be the field generated by
all the exponents appearing in the support of the series in $\bigcup_{m\in\mathbb{N}}\widehat{\mathcal{A}_{m}}$
(we call $\mathbb{K}$ the \emph{field of exponents} of $\mathcal{A}$).
\end{defn}
As mentioned in the introduction, most known polynomially bounded
o-minimal expansions of the reals are generated by some such class.
Let us recall the main examples: 
\begin{itemize}
\item $\mathbb{R}_{\text{an}}$, generated by all real analytic functions
restricted to a compact neighbourhood of the origin \cite{vdd:d}.
\item $\mathbb{R}_{\text{an}^{*}}$, generated by all convergent generalized
power series \cite{vdd:speiss:gen}
\item $\mathbb{R}_{\mathcal{G}}$, generated by all series which are multisummable
in the positive real direction \cite{vdd:speiss:multisum}
\item $\mathbb{R}_{\mathcal{C}\left(M\right)}$, generated by a quasianalytic
Denjoy-Carleman class $\mathcal{C}\left(M\right)$ \cite{rsw}
\item $\mathbb{R}_{\text{an},H}$, generated by a strongly quasianalytic
solution $H$ of a first-order analytic differential equation which
is singular at the origin \cite{rss}
\item $\mathbb{R}_{\mathcal{Q}}$, generated by a collection of functions
(including certain Dulac transition maps) which admit an asymptotic
expansion at the origin in the scale of real nonnegative powers \cite{krs}
\item $\mathbb{R}_{\mathcal{G}^{*}}$, generated by all functions that are
generalized multisummable in the positive real direction\cite{rss:multisummability_generalized}
\end{itemize}
\begin{defn}
\label{def: tr. closed, natural, classical-1} A set $A\subseteq\mathbb{R}$
is \emph{natural} if for every $a\in\mathbb{R}$, the set $A\cap\left(-\infty,a\right)$
is finite. 

A generalized power series $F=\sum_{\alpha\in[0,+\infty)^{N}}c_{\alpha}X^{\alpha}\in\mathbb{R}\left\llbracket X^{*}\right\rrbracket $
is \emph{convergent }if there exists $r>0$ such that $\sum_{\alpha\in[0,+\infty)^{N}}\left|c_{\alpha}\right|r^{\left|\alpha\right|}<+\infty$
(in other words, the family of functions $\left\{ c_{\alpha}x^{\alpha}:\ \alpha\in\text{Supp}\left(F\right)\right\} $
is uniformly summable on $\left[0,r\right]^{N}$). We denote by $\mathbb{R}\left\{ X^{*}\right\} \subseteq\mathbb{R}\left\llbracket X^{*}\right\rrbracket $
the collection of convergent generalized power series.

Let $\mathcal{A}$ be a GQC. We say that:
\begin{itemize}
\item $\mathcal{A}$ is \emph{truncation closed }if, given $f\in\mathcal{A}_{1+m}$
and $\alpha_{0}\in[0,+\infty)$, writing 
\[
\widehat{f}\left(X\right)=\sum_{\alpha\in[0,+\infty)}c_{\alpha}\left(X_{1},\ldots,X_{m}\right)X_{0}^{\alpha},
\]
 we have
\[
\sum_{\alpha<\alpha_{0}}c_{\alpha}\left(X_{1},\ldots,X_{m}\right)X_{0}^{\alpha}\in\widehat{\mathcal{A}_{1+m}}.
\]
By \cite[1.15(3)]{rolin_servi:qeqa} the above property also holds
when replacing $X_{0}$ by another choice of variable $X_{j}$.
\item $\mathcal{A}$ is \emph{natural} if for all $f\in\mathcal{A}_{m}$,
the support of $\widehat{f}$ is contained in a Cartesian product
of $m$ natural subsets of $[0,+\infty)$.
\item $\mathcal{A}=\text{an}^{*}$ if $\forall m\in\mathbb{N},\ \mathcal{A}_{m}=\mathbb{R}\left\{ X_{1}^{*},\ldots,X_{m}^{*}\right\} $.
\item $\mathcal{A}$ is \emph{classical} if $\forall m\in\mathbb{N},\ \widehat{\mathcal{A}_{m}}\subseteq\mathbb{R}\left\llbracket X\right\rrbracket $,
i.e. the image of the quasianalyticity morphism only contains power
series with integer exponents.
\end{itemize}
\end{defn}
All the examples of generalized quasianalytic classes given above
are truncation closed, and all of them are natural, with the exception
of $\mathcal{A}=\text{an}^{*}$. 
\begin{proviso}
\label{proviso}For the rest of the section, we fix a GQC $\mathcal{A}$
with field of exponents $\mathbb{K}$. We suppose that either $\mathcal{A}=\text{an}^{*}$
or $\mathcal{A}$ is natural and truncation closed.

We also fix an ordered group $G$ with $\mathbb{K}$-powers (see Definition
\ref{def: ordered group with powers}) and a group of monomials $\mathcal{M}\subseteq\mathcal{R}$
(see Definition \ref{def: monmial germs}) which is stable under $\mathbb{K}$-powers
(i.e. for all $\mathfrak{m}\in\mathcal{M}$ and all $r\in\mathbb{K}$,
$\mathfrak{m}^{r}\in\mathcal{M}$). Note that such an $\mathcal{M}$
is also an ordered group with $\mathbb{K}$-powers and $\mathcal{A}$
acts on $\mathcal{M}^{<1}$ by composition, whereas $\widehat{\mathcal{A}}$
acts on $G^{<1}$ in the way explained in Remark \ref{rem: action of gen series on groups}.

We suppose furthermore that there exists an ordered group embedding
$\phi:\mathcal{M}\longrightarrow G$ which respects $\mathbb{K}$-powers,
i.e. for all $\mathfrak{m}\in\mathcal{M}$ and all $r\in\mathbb{K}$,
$\phi\left(\mathfrak{m}^{r}\right)=\left(\phi\left(\mathfrak{m}\right)\right)^{r}$.
\end{proviso}
\begin{notation}
\label{notation A}Define $\mathbb{A}$ as $\mathbb{N}$ or $\mathbb{K}_{\geq0}$,
depending on whether $\mathcal{A}$ is a classical or a non-classical
GQC. 

Given a ring $\mathcal{B}$, we denote by $\mathcal{B}^{\times}$
the subset of $\mathcal{B}$ consisting of the invertible elements.
\end{notation}
\begin{lem}
\label{lem: units are units-1}Let $a\in\mathcal{A}_{n}$. Then $a\left(0\right)=c\in\mathbb{R}$
if and only if the constant term of the series $\widehat{a}$ is $c$.

In particular, $a\in\mathcal{A}_{n}^{\times}\Longleftrightarrow\widehat{a}\in\widehat{\mathcal{A}_{n}}^{\times}$.
\end{lem}
\begin{proof}
Suppose first that the constant term of $\widehat{a}$ is zero. 

If $n=1$ then $\widehat{a}\left(X\right)=X^{\alpha}B\left(X\right)$,
for some $\alpha\in\mathbb{A}_{>0}$ and $B\in\mathbb{R}\left\llbracket X^{*}\right\rrbracket $.
By Monomial Division \cite[1.15(2)]{rolin_servi:qeqa}, there exists
$b\in\mathcal{A}_{1}$ such that $a\left(x\right)=x^{\alpha}b\left(x\right)$.
In particular $a\left(0\right)=0$, since $\alpha>0$.

If $n>1$, let $X=\left(X_{1},\ldots,X_{n}\right)$. By the well-order
property of the support of $\widehat{a}$ (see \cite[Lemma 4.8]{vdd:speiss:gen}),
there are $s\in\mathbb{N}^{*}$, exponents $\alpha_{i}\in\mathbb{A}^{n}$,
for $0\leq i\leq s$, and generalized power series $B_{i}\in\mathbb{R}\left\llbracket X^{*}\right\rrbracket $
with nonzero constant term, for $1\leq i\leq s$ (not necessarily
in the image of the morphism $\ \widehat{}\ $), such that $\text{gcd}\left(X^{\alpha_{1}},\ldots,X^{\alpha_{s}}\right)=1$
and 
\begin{equation}
\widehat{a}\left(X\right)=X^{\alpha_{0}}\left(X^{\alpha_{1}}B_{1}\left(X\right)+\cdots+X^{\alpha_{s}}B_{s}\left(X\right)\right).\label{eq: finite presentation-1}
\end{equation}
If $X^{\alpha_{0}}\not=1$ then by Monomial Division \cite[1.15(2)]{rolin_servi:qeqa}
there exists $b\in\mathcal{A}_{n}$ such that $a\left(x\right)=x^{\alpha_{0}}b\left(x\right)$.
In particular, $a\left(0\right)=0$.

If $X^{\alpha_{0}}=1$, then up to permutation we may suppose that
there is $n_{1}<n$ such that $\alpha_{1,j}\not=0\Longleftrightarrow j\leq n_{1}$.
If $X'=\left(X_{1},\ldots,X_{n_{1}}\right)$, then by \cite[1.15(4)]{rolin_servi:qeqa}
we may repeat the argument and write \eqref{eq: finite presentation-1}
for $\widehat{a_{1}}\left(X'\right):=\widehat{a}\left(X',0\right)$:
either we can factor out a common monomial or eventually we reduce
to the case $n=1$.

Now suppose that the constant term of $\widehat{a}$ is $c\not=0$.
Hence $\widehat{b}:=\widehat{a}-c$ has zero constant term. By the
first part of the proof, so does $b:=a-c$. Hence $a\left(0\right)=c$.

Finally, if $a\left(0\right)=c\not=0$ then by the previous paragraph
$\widehat{a-c}$ has zero constant term, hence the constant term of
$\widehat{a}$ is $c$.
\end{proof}

\subsection{Monomialization\label{subsec:Monomialization}}
\begin{notation}
For $B\subseteq\mathbb{K}$ and $\mathfrak{m}_{1},\ldots\mathfrak{m_{\ell}}\in\mathcal{M}$,
we let $\langle\mathfrak{m}_{1},\ldots,\mathfrak{m}_{\ell}\rangle^{B}$
be the multiplicative group generated by $\left\{ \mathfrak{m}_{i}^{\alpha_{i}}:\ i=1,\ldots,\ell,\ \alpha_{i}\in B\right\} $.
\end{notation}
\begin{lem}[Monomialization]
\label{lem: monomialization-1}Let $a\in\mathcal{A}_{\ell}$ and
$\overline{\mathfrak{m}}=\left(\mathfrak{m}_{1},\ldots,\mathfrak{m}_{\ell}\right)\mathrm{\ with\ }\mathfrak{m}_{i}\in\mathcal{M}^{<1}$.
\begin{enumerate}
\item If $\widehat{a}\left(\phi\left(\overline{\mathfrak{m}}\right)\right)\not=0$
then $a\left(\overline{\mathfrak{m}}\right)\not=0$ and there are
$k\in\mathbb{N},\mathfrak{n}_{0}\in\mathcal{M},\overline{\mathfrak{n}}=\left(\mathfrak{n}_{1},\ldots,\mathfrak{n}_{k}\right)\mathrm{\ with\ }\mathfrak{n}_{i}\in\mathcal{M}^{<1}$
and $U\in\mathcal{A}_{k}^{\times}$ such that
\[
a\left(\overline{\mathfrak{m}}\right)=\mathfrak{n}_{0}U\left(\overline{\mathfrak{n}}\right)\mathrm{\ and\ }\widehat{a}\left(\phi\left(\overline{\mathfrak{m}}\right)\right)=\phi\left(\mathfrak{n}_{0}\right)\widehat{U}\left(\phi\left(\overline{\mathfrak{n}}\right)\right).
\]
Moreover, $\mathfrak{n}_{i}\in\langle\mathfrak{m}_{1},\ldots,\mathfrak{m}_{\ell}\rangle^{\mathbb{K}}$
(for $i=0,\ldots,k$).
\item If $\widehat{a}\left(\phi\left(\overline{\mathfrak{m}}\right)\right)=0$
then $a\left(\overline{\mathfrak{m}}\right)=0$.
\end{enumerate}
\end{lem}
\begin{proof}
Let $X=\left(X_{1},\ldots,X_{\ell}\right)$. We may assume that the
$\mathfrak{m}_{i}$ are pairwise distinct.

Similar to the proof of Lemma \ref{lem: units are units-1}, if $\widehat{a}\not=0$
then there are $s\in\mathbb{N}^{*}$, a set $J\subseteq[0,+\infty)^{\ell}$
of cardinality $s$, a collection of monomials $\mathcal{G}=\left\{ X^{\alpha}:\ \alpha\in J\right\} $
which do not divide one another, and series $B_{\alpha}\in\mathbb{R}\left\llbracket X^{*}\right\rrbracket $
with nonzero constant term ($\alpha\in J$), such that
\[
\widehat{a}\left(X\right)=\sum_{\alpha\in J}X^{\alpha}B_{\alpha}\left(X\right).
\]
If $\widehat{a}=0$ we set $\mathcal{G}=\emptyset$. 

Notice that if $\widehat{a}=0$ then by quasianalyticity $a=0$ and
the lemma holds trivially.

Recall from \cite[4.9,4.11]{vdd:speiss:gen} the definition of \emph{blow-up
height} $\text{b}\left(X^{\alpha},X^{\beta}\right)$ of two monomials
and the blow-up height $\text{b}\left(\mathcal{G}\right)$ of a finite
collection of monomials. We extend this definition by setting $\text{b}\left(\mathcal{G}\right)=\left(0,0\right)$
if either $\mathcal{G}=\emptyset$ or $\ell=1$.

Notice that if $\mathcal{G}\not=\emptyset$ and $\text{b}\left(\mathcal{G}\right)=\left(0,0\right)$,
then the monomials $X^{\alpha}$ above are linearly ordered by division,
hence by factoring out the smallest monomial $X^{\alpha_{0}}$ and
by Monomial Division \cite[1.15(2)]{rolin_servi:qeqa} and Lemma \ref{lem: units are units-1},
there exists $U\in\mathcal{A}_{\ell}^{\times}$ such that
\[
\widehat{a}\left(\phi\left(\overline{\mathfrak{m}}\right)\right)=\phi\left(\overline{\mathfrak{m}}^{\alpha_{0}}\right)\widehat{U}\left(\phi\left(\overline{\mathfrak{m}}\right)\right)\not=0\mathrm{\ and\ }a\left(\overline{\mathfrak{m}}\right)=\overline{\mathfrak{m}}^{\alpha_{0}}U\left(\overline{\mathfrak{m}}\right)\not=0,
\]
so the lemma holds. In particular, this proves the lemma when $\ell=1$.

For the general case, we argue by induction on the pairs $\left(\ell,\text{b}\left(\mathcal{G}\right)\right)$,
ordered lexicographically. By the above discussion we may suppose
that $\ell>1$ and $\text{b}\left(\mathcal{G}\right)\not=\left(0,0\right)$.

If there exist $i\not=j\in\left\{ 1,\ldots,\ell\right\} $ and $\alpha_{i},\beta_{j}\in\mathbb{K}^{>0}$
such that $\mathfrak{m}_{i}^{\alpha_{i}}=\mathfrak{m}_{j}^{\beta_{j}}$,
then, up to permutation \cite[1.15(2)]{rolin_servi:qeqa}, we may
suppose $i=1,j=2$. Let $\mathfrak{m}=\mathfrak{m}_{1}^{\frac{1}{\beta_{2}}}\in\mathcal{M}$.
Then, by \cite[1.15(1), Lemma 1.19]{rolin_servi:qeqa},
\[
\widehat{h}\left(Z,X_{3},\ldots,X_{\ell}\right):=\widehat{a}\left(Z^{\beta_{2}},Z^{\alpha_{1}},X_{3},\ldots,X_{\ell}\right)\in\widehat{\mathcal{A}_{\ell-1}}
\]
and $\widehat{h}\left(\mathfrak{m},\mathfrak{m}_{3},\ldots,\mathfrak{m}_{\ell}\right)=\widehat{a}\left(\mathfrak{m}_{1},\ldots,\mathfrak{m}_{\ell}\right)$,
hence we conclude by induction.

Otherwise, there are $X^{\alpha},X^{\beta}\in\mathcal{G}$ such that
$\text{b}\left(X^{\alpha},X^{\beta}\right)$ is nonzero and minimal.
Suppose $\alpha_{1},\beta_{2}\not=0$ and $\mathfrak{m}_{1}^{\alpha_{1}}>\mathfrak{m}_{2}^{\beta_{2}}$.
If $\alpha_{1},\beta_{2}\in\mathbb{N}$ then, up to composing with
a ramification as we did in the previous case, we may suppose that
$\beta_{2}|\alpha_{1}$ (in particular, $\alpha_{1}/\beta_{2}\in\mathbb{A}$)
. Consider the admissible transformation
\[
\nu\left(X\right)=\left(X_{1},X_{1}^{\alpha_{1}/\beta_{2}}X_{2},X_{3},\ldots,X_{\ell}\right),
\]
obtained by composing a ramification and a singular blow-up transformation.
In the series $\widehat{a}\circ\nu$, the term $X^{\alpha}B_{\alpha}\left(X\right)+X^{\beta}B_{\beta}\left(X\right)$
is replaced by $X^{\alpha}B_{\alpha}\left(X\right)+X_{1}^{\alpha_{1}}X^{\beta}B_{\beta}\left(X\right)$
and
\[
\text{b}\left(X^{\alpha},X_{1}^{\alpha_{1}}X^{\beta}\right)=\text{b}\left(X^{\alpha},X^{\beta}\right)-1.
\]
Let $\mathfrak{m}=\mathfrak{m}_{1}^{-\frac{\alpha_{1}}{\beta_{2}}}\mathfrak{m}_{2}\in\mathcal{M}$.
Since $\mathfrak{m}<1$ and 
\[
\widehat{a}\circ\nu\left(\mathfrak{m}_{1},\mathfrak{m},\mathfrak{m}_{3},\ldots,\mathfrak{m}_{\ell}\right)=\widehat{a}\left(\mathfrak{m}_{1},\mathfrak{m}_{2},\mathfrak{m}_{3},\ldots,\mathfrak{m}_{\ell}\right),
\]
we conclude by induction.
\end{proof}

\subsection{Splitting\label{subsec:Splitting}}
\begin{notation}
If
\begin{equation}
\sigma=\sum_{\alpha<\gamma}c_{\alpha}g_{\alpha}\in\mathbb{R}\left(\left(G\right)\right),\label{eq:Hahn series}
\end{equation}
$*\in\left\{ <,=,>\right\} $ and $g\in G$, we set
\[
\sigma_{*g}:=\begin{cases}
{\displaystyle \sum_{\alpha\in\mathcal{S}_{*g}}c_{\alpha}g_{\alpha}} & \mathrm{if}\ \mathcal{S}_{*g}:=\left\{ \alpha<\gamma:\ g_{\alpha}*g\right\} \not=\emptyset\\
0 & \mathrm{otherwise}
\end{cases}.
\]
If $\widehat{a},\widehat{b}\in\mathbb{R}\left(\left(G\right)\right)$
we write $\widehat{a}\subseteq\widehat{b}$ if $\widehat{a}=\sum_{\alpha}a_{\alpha}g_{\alpha}$
is a subseries of $\widehat{b}=\sum_{\alpha}b_{\alpha}g_{\alpha}$,
i.e. if $\text{Supp}\left(\widehat{a}\right)\subseteq\text{Supp}\left(\widehat{b}\right)$
and for all $\alpha\in\text{Supp}\left(\widehat{a}\right),\ a_{\alpha}=b_{\alpha}$. 
\end{notation}
In general, a subseries of a series $\widehat{a}\in\widehat{\mathcal{A}}$
is not necessarily an element of $\widehat{\mathcal{A}}$. The fact
that $\widehat{\mathcal{A}}$ is truncation closed only ensures that
certain subseries (obtained by truncating variable by variable) are
still in $\widehat{\mathcal{A}}$.
\begin{lem}[Splitting]
\label{lem: splitting} Let $a\in\mathcal{A}_{\ell},\mathfrak{m}\in\mathcal{M},*\in\left\{ <,=,>\right\} $
and $\mathfrak{m}_{1},\ldots,\mathfrak{m}_{\ell}\in\mathcal{M}$ such
that $\mathfrak{m}_{i}<1$. Let $g=\phi\left(\mathfrak{m}\right)$
and $g_{i}=\phi\left(\mathfrak{m}_{i}\right)$ for $i=1,\ldots,\ell$.
Then there exists a unique $a_{*g}\in\mathcal{A}_{\ell}$ such that
$\widehat{a_{*g}}\subseteq\widehat{a}$ and
\begin{equation}
\left(\widehat{a}\left(g_{1},\ldots,g_{\ell}\right)\right)_{*g}=\widehat{a_{*g}}\left(g_{1},\ldots,g_{\ell}\right).\label{eq: truncation of series-1}
\end{equation}
\end{lem}
\begin{proof}
Since the case $\mathcal{A}=\text{an}^{*}$ is trivial, we may suppose
that $\mathcal{A}$ is natural and truncation closed. We prove the
case where $*$ is $<$, the other two cases follow easily. 

Notice that if we write $\widehat{a}\left(g_{1},\ldots,g_{\ell}\right)$
as a Hahn series $\sum_{\alpha<\gamma}\widetilde{c}_{\alpha}\widetilde{g}_{\alpha}$,
then $\widetilde{g}_{\alpha}\in\langle g_{1},\ldots,g_{\ell}\rangle^{\mathbb{A}}$.
If the statement holds for $a-a\left(0\right)$ then clearly it holds
for $a$, hence we may suppose that $a\left(0\right)=0$. In particular,
by Lemma \ref{lem: units are units-1} each $\widetilde{g}_{\alpha}<1$,
so if $g\geq1$ then $a_{<g}=a$. Hence we may suppose that $g<1$. 

We prove, by induction on $\ell\geq1$, that $\forall a\in\mathcal{A}_{\ell}$
$\forall g\in\phi\left(\mathcal{M}\right)^{<1}$ and $\forall g_{1},\ldots,g_{\ell}\in\phi\left(\mathcal{M}\right)$
with $g_{1}<\ldots<g_{\ell}<1$, there is $a_{<g}$ as above such
that \eqref{eq: truncation of series-1} holds. Write
\begin{align*}
\widehat{a}\left(X\right) & =\sum_{\mathbf{r}=\left(r_{1},\ldots,r_{\ell}\right)\in[0,+\infty)^{\ell}}c_{\mathbf{r}}X^{\mathbf{r}}\\
 & =\sum_{r_{1}\in[0,+\infty)}B_{r_{1}}\left(X_{2},\ldots,X_{\ell}\right)X_{1}^{r_{1}},
\end{align*}
for some $B_{r_{1}}\in\mathbb{R}\left\llbracket X_{2}^{*},\ldots,X_{\ell}^{*}\right\rrbracket $.
Since $\mathcal{A}$ is truncation closed and by Monomial Division,
we have that there exist $b_{r_{1}}\in\mathcal{A}_{\ell-1}$ such
that $\widehat{b_{r_{1}}}=B_{r_{1}}$, for all $r_{1}$ appearing
as the first coordinate of some element of the support of $\widehat{a}$.

If $\forall r\in[0,+\infty),\ g_{1}^{r}>g$, then, thanks to the order
on the $g_{i}$, we have, for all $\left(r_{1},\ldots,r_{\ell}\right)\in[0,+\infty)^{\ell}$,
that $g_{1}^{r_{1}}\cdots g_{\ell}^{r_{\ell}}\ge g_{1}^{r_{1}+\cdots+r_{\ell}}>g$.
Hence in this case $a_{<g}=0$.

Otherwise, note that, since $g_{i}<1$, if $r\in[0,+\infty)$ is such
that $g_{1}^{r}<g$, then for all $\left(r_{2},\ldots,r_{\ell}\right)\in[0,+\infty)^{\ell-1}$,
\[
g_{1}^{r}g_{1}^{r_{2}}\cdots g_{\ell}^{r_{\ell}}<g.
\]
In particular, if $\forall r\in[0,+\infty),\ g_{1}^{r}<g$, then $a_{<g}=a$. 

Suppose that $s:=\sup\left\{ r\in[0,+\infty):\ g_{1}^{r}\geq g\right\} \in[0,+\infty)$.
Since $\mathcal{A}$ is truncation closed, there is $a_{s}\in\mathcal{A}_{\ell}$
such that
\[
\widehat{a_{s}}\left(X\right)=\sum_{r_{1}>s}\widehat{b_{r_{1}}}\left(X_{2},\ldots,X_{\ell}\right)X_{1}^{r_{1}}.
\]

If $\ell=1$ then we set $a_{<g}=a_{s}$, and we are done. Otherwise,
since $\mathcal{A}$ is natural, there are finitely many $r\leq s$
appearing as the first coordinate of some element of the support of
$\widehat{a}$. For each such $r$, consider the function 
\[
b_{r}\left(x_{2},\ldots,x_{\ell}\right)\in\mathcal{A}_{\ell-1}
\]
and apply the inductive hypothesis to $b_{r}$, replacing $g$ by
$g_{0}:=g_{1}^{-r}g$. It follows that there exists $h_{r}\in\mathcal{A}_{\ell-1}$
such that $\widehat{h_{r}}\subseteq\widehat{b_{r}}$ and 
\[
\left(\widehat{b_{r}}\left(g_{2},\ldots,g_{\ell}\right)\right)_{<g_{0}}=\widehat{h_{r}}\left(g_{2},\ldots,g_{\ell}\right).
\]
We conclude by setting
\[
a_{<g}\left(x\right)=a_{s}\left(x\right)+\sum_{r\leq s}x_{1}^{r}h_{r}\left(x_{2},\ldots,x_{\ell}\right).
\]
\end{proof}

\subsection{Embedding\label{subsec:Embedding}}
\begin{defn}
\label{def: F}Let 
\[
\mathcal{F}=\left\{ \mathfrak{m}_{0}a\left(\mathfrak{m}_{1},\ldots,\mathfrak{m}_{\ell}\right):\ \ell\in\mathbb{N},\mathfrak{m}_{0}\in\mathcal{M},\overline{\mathfrak{m}}=\left(\mathfrak{m}_{1},\ldots,\mathfrak{m}_{\ell}\right)\subseteq\mathcal{M}^{<1},\ a\in\mathcal{A}_{\ell}\right\} \subseteq\mathcal{R}.
\]
\end{defn}
\begin{lem}
\label{lem: F is a fiels}$\mathcal{F}$ is an ordered field (with
respect to the pointwise order) containing $\mathbb{R}$.
\end{lem}
\begin{proof}
It is clear that $\mathcal{F}$ is stable under multiplication and
contains $\mathbb{R}$.

If $f_{i}=\mathfrak{m}_{0,i}a_{i}\left(\overline{\mathfrak{m}_{i}}\right)$
(for $i=1,2$) then suppose wlog $\mathfrak{m}_{0,1}\geq\mathfrak{m}_{0,2}$
and write
\[
f_{1}+f_{2}=\mathfrak{m}_{0,1}\left(a_{1}\left(\overline{\mathfrak{m}_{1}}\right)+\mathfrak{m}_{0,1}^{-1}\mathfrak{m}_{0,2}a_{2}\left(\overline{\mathfrak{m}}_{2}\right)\right),
\]
where the part in brackets is either of the form $a\left(\overline{\mathfrak{m}_{1}},\mathfrak{m}_{0,1}^{-1}\mathfrak{m}_{0,2},\overline{\mathfrak{m}}_{2}\right)$,
if $\mathfrak{m}_{0,1}>\mathfrak{m}_{0,2}$, or of the form $a\left(\overline{\mathfrak{m}_{1}},\overline{\mathfrak{m}}_{2}\right)$,
if $\mathfrak{m}_{0,1}=\mathfrak{m}_{0,2}$, for some $a\in\mathcal{A}$.
Hence $\mathcal{F}$ is stable under addition.

If $f\in\mathcal{F}\setminus\left\{ 0\right\} $ then by Lemma \ref{lem: monomialization-1}
$f$ can be rewritten in monomial form: $\mathfrak{m}_{0}U\left(\overline{\mathfrak{m}}\right)$,
with $U\in\mathcal{A}_{\ell}^{\times}$. In particular, $1/f=\mathfrak{m}_{0}^{-1}U^{-1}\left(\overline{\mathfrak{m}}\right)$.
Hence $\mathcal{F}$ is stable under multiplicative inverse. 

Lemma \ref{lem: monomialization-1} also implies that $f$ has the
same (nonzero) sign as $U\left(0\right)$, hence the field $\mathcal{F}$
is ordered.
\end{proof}
As a consequence of the Monomialization Lemma \ref{lem: monomialization-1}
we obtain the following.
\begin{cor}
\label{cor: monom of F}Let $f\in\mathcal{F}^{\times}$. There are
$k\in\mathbb{N},\mathfrak{n}_{0}\in\mathcal{M},\overline{\mathfrak{n}}=\left(\mathfrak{n}_{1},\ldots,\mathfrak{n}_{k}\right)\subseteq\mathcal{M}^{<1}$
and $U\in\mathcal{A}_{k}^{\times}$ such that $f=\mathfrak{n}_{0}U\left(\overline{\mathfrak{n}}\right)$.
\end{cor}
\begin{proof}
Write $f$ as $\mathfrak{m}_{0}a\left(\overline{\mathfrak{m}}\right)$,
for some $a\in\mathcal{A}$. Since $f\not=0$, we have that $a\left(\overline{\mathfrak{m}}\right)\not=0$,
so by Lemma \ref{lem: monomialization-1} there are $k\in\mathbb{N},\widetilde{\mathfrak{n}_{0}}\in\mathcal{M},\overline{\mathfrak{n}}=\left(\mathfrak{n}_{1},\ldots,\mathfrak{n}_{k}\right)\subseteq\mathcal{M}^{<1}$
and $U\in\mathcal{A}_{k}^{\times}$ such that $a\left(\overline{\mathfrak{m}}\right)=\widetilde{\mathfrak{n}_{0}}U\left(\overline{\mathfrak{n}}\right)$.
Hence the corollary holds, with $\mathfrak{n}_{0}=\widetilde{\mathfrak{n}_{0}}\mathfrak{m}_{0}$.
\end{proof}
\begin{lem}[Composition]
\label{lem: composition}Let $f_{1},\ldots,f_{\ell}\in\mathcal{F}^{\prec1}$
and $a\in\mathcal{A}_{\ell}$. Then $a\circ\left(f_{1},\ldots,f_{\ell}\right)\in\mathcal{F}$.
\end{lem}
\begin{proof}
We may suppose that $f_{i}\not=0$ $\left(i=1,\ldots,\ell\right)$
by \cite[1.8(7)]{rolin_servi:qeqa}, and that $a\circ\left(f_{1},\ldots,f_{\ell}\right)\not=0$.
By Corollary \ref{cor: monom of F}, there are $\mathfrak{n}_{i},\overline{\mathfrak{m}_{i}}\subseteq\mathcal{M}^{<1}$
and units $U_{i}\in\mathcal{A}^{\times}$ $\left(i=1,\ldots,\ell\right)$
such that $f_{i}=\mathfrak{n}_{i}U_{i}\left(\overline{\mathfrak{m}_{i}}\right)$.
Then, arguing as in \cite[Lemma 1.19]{rolin_servi:qeqa}, we can find
$b\in\mathcal{A}$ such that $a\circ\left(f_{1},\ldots,f_{\ell}\right)=b\left(\mathfrak{n}_{1},\overline{\mathfrak{m}_{1}},\ldots,\mathfrak{n_{\ell}},\overline{\mathfrak{m}_{\ell}}\right)$.
\end{proof}
Recall that a real germ $f\in\mathcal{R}$ is differentiable if it
admits a differentiable representative. The hypotheses of the following
proposition are satisfied for example if $\mathcal{F}$ is a differential
subfield of a Hardy field.
\begin{prop}[Derivation]
\label{prop: derivation}Suppose that $\mathcal{M}$ is a group of
differentiable germs such that 
\[
\mathcal{M}':=\left\{ \mathfrak{m}':\ \mathfrak{m}\in\mathcal{M}\right\} \subseteq\mathcal{F}\mathrm{\ and\ }\left(\mathcal{M}^{<1}\right)'\subseteq\mathcal{F}^{\prec1}.
\]
Then every germ $f$ in $\mathcal{F}$ is differentiable and $f'\in\mathcal{F}$.
\end{prop}
\begin{proof}
Let $f=\mathfrak{m}_{0}a\left(\overline{\mathfrak{m}}\right)$ for
some $a\in\mathcal{A}_{\ell}$ and recall from \cite[Remark 1.17]{rolin_servi:qeqa}
that, for $i=1,\ldots,\ell$, $\partial_{i}a\in\mathcal{A}_{\ell}$,
where $\partial_{i}a$ is the germ at zero of $x_{i}\frac{\partial a}{\partial x_{i}}$
(extended by continuity at zero). Then
\[
f'=\mathfrak{m}_{0}'a\left(\overline{\mathfrak{m}}\right)+\mathfrak{m}_{0}\prod_{i=1}^{\ell}\mathfrak{m}_{i}^{-1}\sum_{i=1}^{\ell}\left(\prod_{j\not=i}\mathfrak{m}_{j}\right)\mathfrak{m}_{i}'\partial_{i}a\left(\overline{\mathfrak{m}}\right).
\]
Notice that, by the closure properties of $\mathcal{A}$ listed in
\cite{rolin_servi:qeqa}, there is $b_{i}\in\mathcal{A}$ such that
\[
b_{i}\left(\overline{\mathfrak{m}},\mathfrak{m}_{i}'\right)=\left(\prod_{j\not=i}\mathfrak{m}_{j}\right)\mathfrak{m}_{i}'\partial_{i}a\left(\overline{\mathfrak{m}}\right).
\]
Since $\mathfrak{m}_{i}'\prec1$, the above germ is in $\mathcal{F}$.
\end{proof}
The main result of this section is the following. Recall that $\mathcal{A},G,\mathcal{M}$
and $\phi$ are as in Proviso \ref{proviso}.
\begin{thm}[Embedding]
\label{thm: embedding} The ordered group embedding $\phi:\mathcal{M}\longrightarrow G$
can be extended to a truncation closed ordered field embedding
\[
\phi:\mathcal{F}\longrightarrow\mathbb{R}\left(\left(G\right)\right),
\]
which provides a transasymptotic expansion for the germs in $\mathcal{F}$,
by setting
\begin{equation}
\phi\left(\mathfrak{m}_{0}a\left(\mathfrak{m}_{1},\ldots,\mathfrak{m}_{\ell}\right)\right):=\phi\left(\mathfrak{m}_{0}\right)\widehat{a}\left(\phi\left(\mathfrak{m}_{1}\right),\ldots,\phi\left(\mathfrak{m}_{\ell}\right)\right).\label{eq: phi embedding}
\end{equation}
\end{thm}
\begin{proof}
First, we show that $\phi$ is a well defined field embedding: arguing
as in the proof of Lemma \ref{lem: F is a fiels}, it is enough to
show that, for $f\in\mathcal{F}$, if $f=0$ then $\phi\left(f\right)=0$.
This latter statement follows from the Monomialization Lemma \ref{lem: monomialization-1}(1). 

Lemma \ref{lem: units are units-1} and Corollary \ref{cor: monom of F}
imply that $f$ and $\phi\left(f\right)$ have the same sign (which
is the sign of the constant term of the unit in the monomialized form).
Hence $\phi$ is an ordered field embedding.

To show that $\phi$ is truncation closed, let $f=\mathfrak{m}_{0}a\left(\mathfrak{m}_{1},\ldots,\mathfrak{m}_{\ell}\right)\in\mathcal{F}\setminus\left\{ 0\right\} $
and let $\sigma=\phi\left(f\right)=\phi\left(\mathfrak{m}_{0}\right)\widehat{a}\left(\phi\left(\mathfrak{m}_{1}\right),\ldots,\phi\left(\mathfrak{m}_{\ell}\right)\right)=\sum_{\alpha<\gamma}c_{\alpha}g_{\alpha}$.
Let $\beta<\gamma$. Define $g:=\left(\phi\left(\mathfrak{m}_{0}\right)\right)^{-1}g_{\beta}$.
By the Splitting Lemma \ref{lem: splitting}, there exists a unique
germ $a_{>g}\in\mathcal{A}_{\ell}$ such that $\widehat{a_{>g}}\subseteq\widehat{a}$
and 
\[
\left(\widehat{a}\right)_{>g}\left(\phi\left(\mathfrak{m}_{1}\right),\ldots,\phi\left(\mathfrak{m}_{\ell}\right)\right)=\widehat{a_{>g}}\left(\phi\left(\mathfrak{m}_{1}\right),\ldots,\phi\left(\mathfrak{m}_{\ell}\right)\right).
\]
Then, defining $f_{\beta}=\mathfrak{m}_{0}a_{>g}\left(\mathfrak{m}_{1},\ldots,\mathfrak{m}_{\ell}\right)\in\mathcal{F}$,
we have that $\phi\left(f_{\beta}\right)=\sum_{\alpha<\beta}c_{\alpha}g_{\alpha}$.

It remains to prove that $\phi$ provides a transasymptotic expansion.
Let $f\in\mathcal{F}\setminus\left\{ 0\right\} ,\ \sigma=\phi\left(f\right)=\sum_{\alpha<\gamma}c_{\alpha}g_{\alpha}$
and $\beta<\gamma$. Let $f_{\beta}$ be, as above, the unique germ
such that $\phi\left(f_{\beta}\right)=\sum_{\alpha<\beta}c_{\alpha}g_{\alpha}$
and consider $h_{\beta}:=f-f_{\beta}$. We have that $\phi\left(h_{\beta}\right)=\sum_{\alpha\geq\beta}c_{\alpha}g_{\alpha}$,
hence $\text{lt}\left(\phi\left(h_{\beta}\right)\right)=c_{\beta}g_{\beta}$.
We aim to prove that $h_{\beta}\sim_{+\infty}c_{\beta}\phi^{-1}\left(g_{\beta}\right)$.
To see this, apply the Monomialization Lemma \ref{lem: monomialization-1}
to $h_{\beta}$ and $\phi\left(h_{\beta}\right)$: there are monomials
$\mathfrak{m}_{\beta}\in\mathcal{M},\overline{\mathfrak{n}_{\beta}}\in\left(\mathcal{M}^{<1}\right)^{\ell}$
and a unit $U_{\beta}\in\mathcal{A}_{\ell}^{\times}$ such that $h_{\beta}=\mathfrak{m}_{\beta}U_{\beta}\left(\overline{\mathfrak{n}_{\beta}}\right)$
(hence $h_{\beta}\sim_{+\infty}U_{\beta}\left(0\right)\mathfrak{m}_{\beta}$)
and $\phi\left(h_{\beta}\right)=\phi\left(\mathfrak{m}_{\beta}\right)\widehat{U}_{\beta}\left(\phi\left(\overline{\mathfrak{n}_{\beta}}\right)\right)$
(hence $\text{lt}\left(\phi\left(h_{\beta}\right)\right)=\widehat{U}_{\beta}\left(0\right)\phi\left(\mathfrak{m}_{\beta}\right)$).
By Lemma \ref{lem: units are units-1}, it follows that $U_{\beta}\left(0\right)=c_{\beta}$
and $\mathfrak{m}_{\beta}=\phi^{-1}\left(g_{\beta}\right)$.
\end{proof}
\begin{defn}
\label{def: F>M}Let $\mathcal{M}_{0}$ be a subset of $\mathcal{M}$
and $*\in\left\{ \prec,\succ,\preceq,\succeq\right\} $. Define
\[
\mathcal{F}^{*\mathcal{M}_{0}}=\left\{ f\in\mathcal{F}:\ \mathfrak{m}*\mathfrak{m}_{0},\ \forall\mathfrak{m}\in\mathcal{M}\ \text{s.t.}\ \phi\left(\mathfrak{m}\right)\in\text{Supp}\left(\phi\left(f\right)\right),\forall\mathfrak{m}_{0}\in\text{\ensuremath{\mathcal{M}_{0}}}\right\} .
\]
If $\mathcal{M}_{0}=\left\{ \mathfrak{m}_{0}\right\} $ we write $\mathcal{F}^{*\mathfrak{m}_{0}}$
instead of $\mathcal{F}^{*\left\{ \mathfrak{m}_{0}\right\} }$.
\end{defn}
\begin{rem}
\label{rem: F<1}By Corollary \ref{cor: monom of F} and Theorem \ref{thm: embedding},
for $f\in\mathcal{F}$ and $*\in\left\{ \prec,\succ,\preceq,\succeq\right\} $
we have that $f*1\Longleftrightarrow\text{lm}\left(\phi\left(f\right)\right)*1$.
In particular,
\[
\mathcal{F}^{\prec1}=\left\{ f\in\mathcal{F}:\ f\prec1\right\} .
\]
\end{rem}
As a consequence of the Splitting Lemma \ref{lem: splitting} we obtain
the following.
\begin{cor}
\label{cor: splitting f F}For all $f\in\mathcal{F}$ there exist
unique $f_{\succ}\in\mathcal{F}^{\succ1}$ and $f_{\preceq}\in\mathcal{F}^{\preceq1}$
such that
\[
f=f_{\succ}+f_{\preceq}.
\]
\end{cor}
\begin{proof}
Write $f=\mathfrak{m}_{0}U\left(\overline{\mathfrak{m}}\right)$,
with $U\in\mathcal{A}^{\times}$, and let $U_{>\phi\left(\mathfrak{m}_{0}^{-1}\right)}\in\mathcal{A}$
be as in the Splitting Lemma \ref{lem: splitting}. Set $f_{\succ}=\mathfrak{m}_{0}U_{>\phi\left(\mathfrak{m}_{0}^{-1}\right)}\left(\overline{\mathfrak{m}}\right)$
and $f_{\preceq}=f-f_{\succ}$.
\end{proof}

\section{The Main Theorem\label{sec:The-Main-Theorem}}

\subsection{The exponential closure of a GQC\label{subsec:The-exponential-closure}}

Let $\mathcal{A}$ be a GQC with field of exponents $\mathbb{K}$.
Recall from \cite[Definition 1.21]{rolin_servi:qeqa} that $\mathcal{L}_{\mathcal{A}}$
is the language of ordered rings augmented by a function symbol for
every function in $\mathcal{A}_{m,n,\mathbf{1}}$ (where $m,n\in\mathbb{N}$
and \textbf{$\mathbf{1}=\left(1,\ldots,1\right)\in\left(0,+\infty\right)^{m+n}$}).
The structure $\mathbb{R}_{\mathcal{A}}$ is the expansion of the
real ordered field where we interpret the function symbol for $f\in\mathcal{A}_{m,n,\mathbf{1}}$
as the function $f\restriction[0,1)^{m}\times\left(-1,1\right)^{n}$,
extended by zero outside the domain of $f$. By \cite[Theorems A and B]{rolin_servi:qeqa},
$\mathbb{R}_{\mathcal{A}}$ is o-minimal, polynomially bounded (with
field of exponents $\mathbb{K}$) and admits quantifier elimination
in the expansion of $\mathcal{L}_{\mathcal{A}}$ by symbols for the
multiplicative inverse $\left(\cdot\right)^{-1}$ and $n$th-roots
$\sqrt[n]{\cdot}$ $\left(n\in\mathbb{N}^{*}\right)$. If $\exp\restriction\left[0,1\right]\in\mathcal{A}_{1}$,
then \cite[Theorem B]{vdd:speiss:multisum} implies that the natural
expansion $\mathbb{R}_{\mathcal{A},\exp}:=\langle\mathbb{R}_{\mathcal{A}},\exp\rangle$
of $\mathbb{R}_{\mathcal{A}}$ by the unrestricted exponential function
admits quantifier elimination in the language $\mathcal{L}_{\mathcal{A},\exp}:=\mathcal{L}_{\mathcal{A}}\cup\left\{ \text{Exp},\text{Log}\right\} $.
We can actually say more: $\mathbb{R}_{\mathcal{A},\exp}$ satisfies
the following strong form of quantifier elimination.

\begin{thm}[Strong QE for $\mathbb{R}_{\mathcal{A},\exp}$]
\label{thm: strong QE-1}Let $\mathcal{A}$ be a GQC such that $\exp\restriction\left[0,1\right]\in\mathcal{A}_{1}$,
$D\subseteq\mathbb{R}^{N}$ and $\eta:D\longrightarrow\mathbb{R}$
be an $\mathbb{R}_{\mathcal{A},\exp}$-definable function. Then there
exist finitely many terms $t_{1},\ldots,t_{M}$ of the language $\mathcal{L}_{\mathcal{A},\exp}$
such that
\[
\forall x\in D\ \exists i\in\left\{ 1,\ldots,M\right\} ,\ \eta\left(x\right)=t_{i}\left(x\right).
\]
\end{thm}
\begin{proof}
Let $T^{*}$ be a Skolemization of $\text{Th}\left(\mathbb{R}_{\mathcal{A}}\right)$
and let $\mathcal{L}^{*}$ be its language. By a routine compactness
argument (see for example the proof of \cite[Cor. 2.15]{dmm:exp}),
\cite[Theorem B]{vdd:speiss:multisum} implies the statement with
$\mathcal{L}^{*}$ in place of $\mathcal{L}_{\mathcal{A}}$. Then
we conclude by \cite[Corollary 4.3]{rolin_servi:qeqa}, since the
functions $\left(\cdot\right)^{-1},\sqrt[n]{\cdot}$ can be expressed
as $\mathcal{L}_{\mathcal{A},\exp}$-terms. 
\end{proof}
\begin{defn}
\label{def: L_A,exp-terms-1}We denote by $\mathcal{H}\left(\mathbb{R}_{\mathcal{A},\exp}\right)$
the Hardy field whose elements are the germs at $+\infty$ of the
unary functions definable in $\mathbb{R}_{\mathcal{A},\exp}$. Let
\[
\mathcal{T}\left(\mathcal{L}_{\mathcal{A},\exp}\right)=\left\{ t:\ t\text{ is an }\mathcal{L}_{\mathcal{A},\exp}\text{-term in at most }1\text{ free variable}\right\} .
\]
For $t\in\mathcal{T}\left(\mathcal{L}_{\mathcal{A},\exp}\right)$,
let $t^{\mathcal{H}}\in\mathcal{H}\left(\mathbb{R}_{\mathcal{A},\exp}\right)$
be the germ at $+\infty$ of the definable function given by the interpretation
of $t$ in $\mathbb{R}_{\mathcal{A},\exp}$.
\end{defn}
As an immediate consequence of Theorem \ref{thm: strong QE-1} we
obtain the following result.
\begin{cor}
\label{cor: strong QE for H-1}Let $\mathcal{A}$ be a GQC such that
$\exp\restriction\left[0,1\right]\in\mathcal{A}_{1}$. Then for every
$h\in\mathcal{H}\left(\mathbb{R}_{\mathcal{A},\exp}\right)$ there
exists $t\in\mathcal{T}\left(\mathcal{L}_{\mathcal{A},\exp}\right)$
such that $t^{\mathcal{H}}=h$.\hfill{}$\qed$
\end{cor}
For the rest of the section, we fix a GQC $\mathcal{A}$ with field
of exponents $\mathbb{K}$ and such that $\exp\restriction\left[0,1\right]\in\mathcal{A}_{1}$.
We suppose that either $\mathcal{A}=\text{an}^{*}$ or $\mathcal{A}$
is natural and truncation closed.

The main result of this paper, which is proved in Subsection \ref{subsec:The-Main-Theorem-1},
is the following.
\begin{namedthm}
{Main Theorem} There is a truncation closed ordered differential
field embedding
\[
\phi:\mathcal{H}\left(\mathbb{R}_{\mathcal{A},\exp}\right)\longrightarrow\mathbb{T},
\]
which provides a transasymptotic expansion for the germs in $\mathcal{H}\left(\mathbb{R}_{\mathcal{A},\exp}\right)$
and which is also an embedding of $\mathcal{L}_{\mathcal{A},\exp}$-structures.
\end{namedthm}

\subsection{Transseries\label{subsec:Transseries-1}}

Let $Y$ be a single variable, which we stipulate to be larger than
any real number. The Hahn field $\mathbb{R}\left(\left(Y^{\mathbb{R}}\right)\right)$
with monomials in the ordered multiplicative group $\left\{ Y^{r}:\ r\in\mathbb{R}\right\} $
and coefficients in $\mathbb{R}$ is just the field of fractions of
the ring $\mathbb{R}\left\llbracket \left(Y^{-1}\right)^{*}\right\rrbracket $. 

The field $\mathbb{T}$ of transseries, also denoted $\mathbb{R}\left(\left(Y^{\mathbb{R}}\right)\right)^{\text{LE}}$,
is constructed from $\mathbb{R}\left(\left(Y^{\mathbb{R}}\right)\right)$
by adding exponentials and logarithms in a suitable way. We refer
the reader to the construction in \cite[Section 2.8]{vdd_macintyre_marker:log-exp_series}.
Recall in particular that $\mathbb{T}$ is a proper subfield of the
Hahn field $\mathbb{R}\left(\left(G^{LE}\right)\right)$, where $G^{LE}$
is the ordered multiplicative group of $LE$-monomials. 

Moreover, $\mathbb{T}$ has a natural $\left\{ 0,1,-,+,\cdot,\text{Exp,Log}\right\} $-structure,
which is an elementary extension of the real exponential field \cite[Corollary 2.8]{dmm:series}.

Next, we turn $\mathbb{T}$ into an $\mathcal{L}_{\mathcal{A},\exp}$-structure:
for $m,n\in\mathbb{N},\ \sigma_{1},\ldots,\sigma_{m+n}\in\mathbb{T}$
and $f\in\mathcal{A}_{m,n,\mathbf{1}}$, we interpret in $\mathbb{T}$
the function symbol for $f$ evaluated at $\left(\sigma_{1},\ldots,\sigma_{m+n}\right)$
as the transseries $\widehat{f}\left(\sigma_{1},\ldots,\sigma_{m+n}\right)$,
if $\text{lm}\left(\sigma_{i}\right)<1$ for all $i=1,\ldots,m+n$,
and zero otherwise.

Hence, there is a natural map from $\mathcal{T}\left(\mathcal{L}_{\mathcal{A},\exp}\right)$
to $\mathbb{T}$ associating to a term $t$ the interpretation $t^{\mathbb{T}}$
of $t$ in $\mathbb{T}$ as a unary function, evaluated at the transseries
$Y$.
\begin{example}
\label{exa: tilde and hat-1}If $\underline{f}$ is a binary function
symbol in $\mathcal{L}_{\mathcal{A}}$, representing the function
$f\in\mathcal{A}_{2,0,\mathbf{1}}$, and $v$ is a meta-variable,
consider the term
\[
t\left(v\right)=\text{Exp}\left(\text{Exp}\left(\left(1+1\right)\cdot v\right)\cdot\underline{f}\left(\text{Exp}\left(-\text{Log}\left(v\right)\right)\right),\text{Exp}\left(-v\right)\right).
\]
Then $t^{\mathcal{H}}$ is the germ at $+\infty$ of the definable
function $y\longmapsto\text{e}^{\text{e}^{2y}f\left(y^{-1},\text{e}^{-y}\right)}$
and $t^{\mathbb{T}}$ is the transseries $\text{e}^{\text{e}^{2Y}\widehat{f}\left(Y^{-1},\text{e}^{-Y}\right)}$,
where $\widehat{f}\in\mathbb{R}\left\llbracket X_{1}^{*},X_{2}^{*}\right\rrbracket $
is the generalized power series associated to the function $f$ by
the quasianalyticity morphism $\ \widehat{}\ $ in \ref{eq:qa morphism hat-1}.

Notice that it is not easy to read off $t^{\mathbb{T}}$ its expression
as a Hahn series, and in particular its leading term. 
\end{example}

\subsection{$\mathcal{A}$-transseries\label{subsec:Monomialized-germs}}

In this section we construct a differential subfield $\mathcal{F}_{\mathcal{A}}$
of $\mathcal{H}\left(\mathbb{R}_{\mathcal{A},\exp}\right)$ and a
truncation closed ordered differential field embedding
\[
\phi:\mathcal{F}_{\mathcal{A}}\longrightarrow\mathbb{T},
\]
which provides a transasymptotic expansion for the germs in $\mathcal{F}_{\mathcal{A}}$.
The construction is inspired by \cite[Section 2]{kaiser_speissegger:analytic_continuation}.

We start by constructing, inductively for $n\geq-1$, a group of monomials
$\mathcal{M}_{n}\subseteq\mathcal{H}\left(\mathbb{R}_{\mathcal{A},\exp}\right)$
(see Definition \ref{def: monmial germs}) stable under $\mathbb{K}$-powers
and an ordered group embedding $\phi_{n}:\mathcal{M}_{n}\longrightarrow G^{LE}$
which respects $\mathbb{K}$-powers (see Proviso \ref{proviso}).
We then apply the Embedding Theorem \ref{thm: embedding} to extend
$\phi_{n}$ to a truncation closed ordered field embedding $\phi_{n}:\mathcal{F}_{n}\longrightarrow\mathbb{R}\left(\left(G^{LE}\right)\right)$,
where $\mathcal{F}_{n}$ is constructed from $\mathcal{M}_{n}$ as
$\mathcal{F}$ is constructed from $\mathcal{M}$ in Definition \ref{def: F}.
We prove furthermore that $\mathcal{F}_{n}$ is a differential subfield
of $\mathcal{H}\left(\mathbb{R}_{\mathcal{A},\exp}\right)$, that
$\phi_{n}\left(\mathcal{F}_{n}\right)\subseteq\mathbb{T}$ and that
$\phi_{n}$ respects derivatives.
\begin{itemize}
\item Let $\mathcal{M}_{-1}=\left\{ 1\right\} $ and $\phi_{-1}:1\longmapsto1\in G^{LE}$.
Then $\mathcal{F}_{-1}=\mathbb{R}$ and $\phi$ is the identity map.
\item Let $\mathcal{M}_{0}=\left\{ y^{r}:\ r\in\mathbb{K}\right\} \subseteq\mathcal{H}\left(\mathbb{R}_{\mathcal{A},\exp}\right)$
and $\phi_{0}:y^{r}\longmapsto Y^{r}\in G^{LE}$. Clearly, $\mathcal{M}_{0}$
is a group of monomials stable under $\mathbb{K}$-powers and $\phi_{0}$
is an ordered group embedding which respects $\mathbb{K}$-powers.
By the Embedding Theorem \ref{thm: embedding}, $\phi$ extends to
$\mathcal{F}_{0}=\left\{ y^{r_{0}}a\left(y^{r_{1}},\ldots,y^{r_{\ell}}\right):\ \ell\in\mathbb{N}^{*},r_{0}\in\mathbb{K},r_{1},\ldots,r_{\ell}\in\mathbb{K}^{<0},a\in\mathcal{A}_{\ell}\right\} $.
By Proposition \ref{prop: derivation} $\mathcal{F}_{0}$ is a differential
field and by construction $\phi$ is a differential field embedding.
Moreover it is clear that $\phi\left(\mathcal{F}_{0}\right)\subseteq\mathbb{T}$.\\
Define $\mathcal{F}_{0}^{>1}=\mathcal{F}_{0}^{\succ1}\cup\left\{ 0\right\} $.
It is an ordered additive $\mathbb{K}$-vector space.
\item Suppose we have already defined, for $k\leq n$, groups of monomials
$\mathcal{M}_{k}\subseteq\mathcal{H}\left(\mathbb{R}_{\mathcal{A},\exp}\right)$
stable under $\mathbb{K}$-powers, ordered differential field embeddings
$\phi_{k}:\mathcal{F}_{k}\longrightarrow\mathbb{T}$ and ordered $\mathbb{K}$-vector
fields $\mathcal{F}_{k}^{>\mathcal{M}_{k-1}}\subseteq\mathcal{F}_{k}$.\\
Let $\mathcal{M}_{n+1}=\mathcal{M}_{n}\exp\left(\mathcal{F}_{n}^{>\mathcal{M}_{n-1}}\right)$
and, given a monomial $\mathfrak{m}=\mathfrak{m}_{0}\text{e}^{f}\in\mathcal{M}_{n+1}$
(with $\mathfrak{m}_{0}\in\mathcal{M}_{n}$ and $f\in\mathcal{F}_{n}^{>\mathcal{M}_{n-1}}$),
set $\phi_{n+1}\left(\mathfrak{m}\right)=\phi_{n}\left(\mathfrak{m}_{0}\right)\text{e}^{\phi_{n}\left(f\right)}$.\\
We prove that $\mathcal{M}_{n+1}$ is a group of monomials stable
under $\mathbb{K}$-powers: the only nontrivial statement is that
every $\mathfrak{m}=\mathfrak{m}_{0}\text{e}^{f}\in\mathcal{M}_{n+1}$
is a monomial, i.e. $\mathfrak{m}<1\Longrightarrow\mathfrak{m}\prec1$.
For this, write $f=\mathfrak{n}_{0}U\left(\overline{\mathfrak{n}}\right)$,
where $U\in\mathcal{A}^{\times}$ and $\mathfrak{n}_{0},\overline{\mathfrak{n}}\subseteq\mathcal{M}_{n}$.
Since $f\in\mathcal{F}_{n}^{>\mathcal{M}_{n-1}},\ \mathfrak{n}_{0}>\mathcal{M}_{n-1}$.
Write $\mathfrak{m}_{0}=\mathfrak{m}_{1}\text{e}^{f_{1}}$, with $\mathfrak{m}_{1}\in\mathcal{M}_{n-1}\prec\mathfrak{n}_{0}$
and $f_{1}=\mathfrak{n}_{1}U_{1}\left(\overline{\mathfrak{w}}\right)\asymp\mathfrak{n}_{1}\in\mathcal{M}_{n-1}\prec\mathfrak{n}_{0}$.
From $\mathfrak{m}<1$ we deduce that $0>\log\mathfrak{m}=\log\mathfrak{m}_{1}+f_{1}+f$.
Now, if $\mathfrak{m}_{1}\succ1$ then $\log\mathfrak{m}_{1}\prec\mathfrak{m}_{1}\prec\mathfrak{n}_{0}\asymp f$
and if $\mathfrak{m}_{1}\prec1$ then $-\log\mathfrak{m}_{1}=\log\mathfrak{m}_{1}^{-1}\prec\mathfrak{m}_{1}^{-1}\in\mathcal{M}_{n-1}\prec\mathfrak{n}_{0}\asymp f$.
Hence $f\succ\log\mathfrak{m}_{1}+f_{1}$ and $0<\log\mathfrak{m}\asymp\mathfrak{n}_{0}\succ\mathcal{M}_{n-1}\succeq1$.
So $\log\mathfrak{m}$ tends to $-\infty$ and thus $\mathfrak{m}$
tends to zero.\\
We construct $\mathcal{F}_{n+1}$ from $\mathcal{M}_{n+1}$ as $\mathcal{F}$
is constructed from $\mathcal{M}$ in Definition \ref{def: F}. \\
It is clear that $\phi_{n+1}$ is an ordered group embedding respecting
$\mathbb{K}$-powers, hence we apply the Embedding Theorem \ref{thm: embedding}
to extend $\phi_{n+1}$ to the ordered field $\mathcal{F}_{n+1}$.\\
We prove that $\mathcal{F}_{n+1}$ is a differential subfield of $\mathcal{H}\left(\mathbb{R}_{\mathcal{A},\exp}\right)$:
since $\left(\mathfrak{m}\text{e}^{f}\right)'=\mathfrak{m}'\text{e}^{f}+\mathfrak{m}f'\text{e}^{f}\in\mathcal{F}_{n+1}$
and since clearly $\mathcal{F}_{n+1}$ is a subfield of the Hardy
field $\mathcal{H}\left(\mathbb{R}_{\mathcal{A},\exp}\right)$, we
may apply Proposition \ref{prop: derivation}.\\
It is clear by construction that $\phi_{n+1}$ respects derivatives
and that $\phi_{n+1}\left(\mathcal{F}_{n+1}\right)\subseteq\mathbb{T}$.\\
Finally, we let $\mathcal{F}_{n+1}^{>\mathcal{M}_{n}}=\mathcal{F}_{n+1}^{\succ\mathcal{M}_{n}}\cup\left\{ 0\right\} $
(see Definition \ref{def: F>M}).
\end{itemize}
It is easy to see by induction on $n$ that, for $n\geq1$,
\[
\mathcal{F}_{n}^{>1}:=\left\{ f\in\mathcal{F}_{n}:\ \text{Supp}\left(\phi_{n}\left(f\right)\right)>1\right\} =\mathcal{F}_{0}^{>1}\oplus\mathcal{F}_{1}^{>\mathcal{M}_{0}}\oplus\cdots\oplus\mathcal{F}_{n}^{>\mathcal{M}_{n-1}}
\]
and
\begin{equation}
\mathcal{M}_{n+1}=\mathcal{M}_{0}\exp\left(\mathcal{F}_{n}^{>1}\right).\label{eq: F>1}
\end{equation}

\medskip{}

Next, we introduce iterated logarithms: for $k\in\mathbb{N}$, let
$\log_{k+1}=\log\circ\log_{k}$ . Define $\mathcal{M}_{n,k}=\mathcal{M}_{n}\circ\log_{k},\ \mathcal{F}_{n,k}=\mathcal{F}_{n}\circ\log_{k}$
and $\phi_{n,k}:\mathcal{F}_{n,k}\longrightarrow\mathbb{T}$ by $\phi_{n,k}\left(f\circ\log_{k}\left(y\right)\right)=\phi_{n}\left(f\right)\circ\log_{k}\left(Y\right)$.
As right-composition by $\log_{k}$ acts as a change of variables
(both for germs and for transseries) and because $\left(\log_{k}y\right)'\in\mathcal{M}_{1,k-1}$,
$\phi_{n,k}$ is still a truncation closed ordered differential field
embedding, which provides a transasymptotic expansion for the germs
in $\mathcal{F}_{n,k}$ in a scale given by monomials in $\mathcal{M}_{n,k}$.

Finally, define $\mathcal{M}_{\mathcal{A}}=\bigcup_{n,k}\mathcal{M}_{n,k},\ \mathcal{F}_{\mathcal{A}}=\bigcup_{n,k}\mathcal{F}_{n,k}$
and $\phi=\bigcup_{n,k}\phi_{n,k}$. Then by the above, $\mathcal{M}_{\mathcal{A}}$
is a group of monomials stable under $\mathbb{K}$-powers, $\mathcal{F}_{\mathcal{A}}$
is a differential subfield of $\mathcal{H}\left(\mathbb{R}_{\mathcal{A},\exp}\right)$
and $\phi$ is a truncation closed ordered differential field embedding
of $\mathcal{F}_{\mathcal{A}}$ into $\mathbb{T}$ (mapping $\phi\left(\mathcal{M}_{\mathcal{A}}\right)$
into $G^{LE}$) which respects $\mathbb{K}$-powers and which provides
a transasymptotic expansion in a scale of monomials in $\mathcal{M}_{\mathcal{A}}$.
In particular, Lemma \ref{lem: composition} and Corollaries \ref{cor: monom of F}
and \ref{cor: splitting f F} apply to $\mathcal{F}_{\mathcal{A}}$.
Furthermore, as above we have
\begin{equation}
\mathcal{F}_{n,k}^{>1}=\mathcal{F}_{0,k}^{>1}\oplus\mathcal{F}_{1,k}^{>\mathcal{M}_{0,k}}\oplus\cdots\oplus\mathcal{F}_{n,k}^{>\mathcal{M}_{n-1,k}}\ \text{and\ }\mathcal{M}_{n+1,k}=\mathcal{M}_{0,k}\exp\left(\mathcal{F}_{n,k}^{>1}\right).\label{eq: F_k >1}
\end{equation}

We call $\mathbb{T}_{\mathcal{A}}:=\phi\left(\mathcal{F}_{\mathcal{A}}\right)\subseteq\mathbb{T}$
the field of $\mathcal{A}$-transseries.

\subsection{Proof of the Main Theorem\label{subsec:The-Main-Theorem-1}}

In this section we prove that actually $\mathcal{F}_{\mathcal{A}}=\mathcal{H}\left(\mathbb{R}_{\mathcal{A},\exp}\right)$,
hence completing the proof of the Main Theorem.
\begin{prop}
\label{prop F=00003DH}For every $t\in\mathcal{T}\left(\mathcal{L}_{\mathcal{A},\exp}\right)$,
$t^{\mathcal{H}}\in\mathcal{F}_{\mathcal{A}}$.
\end{prop}
\begin{proof}
First let $t$ be a term of the language $\mathcal{L}_{\mathcal{A}}\cup\left\{ \text{Exp}\right\} $.
We argue by induction on the complexity of $t$ as a term that $t^{\mathcal{H}}\in\bigcup_{n}\mathcal{F}_{n}$.

If $t=t_{1}+t_{2}$ or $t=t_{1}\cdot t_{2}$, then the statement follows
by induction from the definition of $t^{\mathcal{H}}$ and the fact
that $\bigcup_{n}\mathcal{F}_{n}$ is a field. If $t=\underline{a}\left(t_{1},\ldots,t_{\ell}\right)$
with $a\in\mathcal{A}_{\ell}$ and $t_{i}^{\mathcal{H}}\in\left(\bigcup_{n}\mathcal{F}_{n}\right)^{<1}$
then conclude by the Composition Lemma \ref{lem: composition}. If
$t=\exp\left(t_{0}\right)$ with $f:=t_{0}^{\mathcal{H}}\in\mathcal{F}_{n}$,
for some $n\in\mathbb{N}$, then apply Corollary \ref{cor: splitting f F}
to write $f=f_{\succ}+f_{\preceq}$. Note that $f_{\succ}\in\mathcal{F}_{n}^{>1}$,
so by \eqref{eq: F>1}, $\exp\left(f_{\succ}\right)\in\mathcal{M}_{n+1}$.
Now write $f_{\preceq}=r+f_{\prec}$, with $r\in\mathbb{R}^{\times}$
and $f_{\prec}\in\mathcal{F}_{n}^{\prec1}$ and notice that, by the
Composition Lemma \ref{lem: composition}, $\exp\left(f_{\preceq}\right)=\text{e}^{r}\sum_{i\geq0}\frac{\left(f_{\prec}\right)^{i}}{i!}\in\mathcal{F}_{n}$.

Finally, to handle the logarithm, it is enough to notice that if $t\in\mathcal{T}\left(\mathcal{L}_{\mathcal{A},\exp}\right)$
is such that $t^{\mathcal{H}}\in\mathcal{F}_{n,k}$ for some $n,k\in\mathbb{N}$,
then $\log\left(t^{\mathcal{H}}\right)\in\mathcal{F}_{n,k+1}$. To
see this, apply Corollary \ref{cor: monom of F} to write $t^{\mathcal{H}}=\mathfrak{m}_{0}U\left(\overline{\mathfrak{m}}\right)$,
with $\mathfrak{m}_{0},\overline{\mathfrak{m}}\subseteq\mathcal{M}_{n,k}$
and $u\in\mathcal{A}^{\times}$. Notice that $\log\circ U\in\mathcal{A}$
and, using \eqref{eq: F_k >1}, $\log\left(\mathfrak{m}_{0}\right)\in\mathcal{F}_{n-1,k+1}$,
hence
\[
\log\left(t^{\mathcal{H}}\right)=\log\left(\mathfrak{m}_{0}\right)+\log\circ U\left(\overline{\mathfrak{m}}\right)\in\mathcal{F}_{n,k+1},
\]
as claimed.
\end{proof}
The Main Theorem now follows from Corollary \ref{cor: strong QE for H-1}
(it is clear from the definition that $\phi$ is an embedding of $\mathcal{L}_{\mathcal{A},\exp}$-structures).
Note that the assumption that $\exp\restriction\left[0,1\right]\in\mathcal{A}_{1}$
is needed in Corollary \ref{cor: strong QE for H-1}.

\section{Applications to non-definability results\label{sec:Applications-to-non-definability-1}}

Recall that $\mathbb{T}_{\mathcal{A}}=\phi\left(\mathcal{H}\left(\mathbb{R}_{\mathcal{A},\exp}\right)\right)\subseteq\mathbb{T}$.
\begin{prop}
\label{prop: zeta not def}Let $\mathcal{A}$ be a classical GQC such
that $\exp\restriction\left[0,1\right]\in\mathcal{A}_{1}$. Then $\zeta\restriction\left(1,+\infty\right)$
is not definable in $\mathbb{R}_{\mathcal{A},\exp}$.
\end{prop}
\begin{proof}
It is easy to see from the construction of $\mathcal{F}_{\mathcal{A}}$
that if $\mathcal{A}$ is classical, then $\mathbb{T}_{\mathcal{A}}$
is a collection of \emph{grid-based} (see \cite[Section 2]{vdh:transs_diff_alg})
transseries. In particular, the support of every $\sigma\in\mathbb{T}_{\mathcal{A}}\cap\mathbb{R}\left\llbracket \left(Y^{-1}\right)^{*}\right\rrbracket $
is a finitely generated $\mathbb{N}$-lattice. Now, if $f\left(y\right)=\zeta\circ\log\left(y\right)$,
then $\widehat{f}\left(Y\right)=\sum_{n}Y^{-\log n}$ (because $f\left(y^{-1}\right)$
is a convergent generalized power series, hence $f$ coincides with
its transasymptotic expansion $\widehat{f}$), whose support is not
finitely generated.
\end{proof}
In particular, we just gave another proof of the fact that $\zeta\restriction\left(1,+\infty\right)$
is not definable in $\mathbb{R}_{\mathcal{G},\exp}$ \cite[Corollary 10.11]{vdd:speiss:multisum}.
\begin{lem}
\label{lem: image of (gen) series}Let $\mathcal{A}$ be a GQC.
\begin{enumerate}
\item If $\mathcal{A}$ is classical, then $\mathbb{T}_{\mathcal{A}}\cap\mathbb{R}\left\llbracket Y^{-1}\right\rrbracket =\widehat{\mathcal{A}_{1}}$.
\item If $\mathcal{A}$ is not classical, then $\mathbb{T}_{\mathcal{A}}\cap\mathbb{R}\left\llbracket \left(Y^{-1}\right)^{*}\right\rrbracket =\widehat{\mathcal{A}_{1}}$.
\end{enumerate}
\end{lem}
\begin{proof}
It is clear from the construction of $\mathcal{F}_{\mathcal{A}}$
that $\mathbb{T}_{\mathcal{A}}\cap\mathbb{R}\left\llbracket \left(Y^{-1}\right)^{\left(*\right)}\right\rrbracket \subseteq\phi\left(\mathcal{F}_{0}\right)$
(where we omit $*$ if $\mathcal{A}$ is a classical quasianalytic
class).
\end{proof}
\begin{cor}
\label{cor: applications}$\ $
\begin{enumerate}
\item $\Gamma\restriction\left(0,+\infty\right)$ is not definable in $\mathbb{R}_{\mathrm{an}^{*},\exp}$.
\item Recall the definition of $\mathbb{R}_{\mathcal{G}^{*}}$ in \cite{rss:multisummability_generalized}.
If $f\in\mathcal{H}\left(\mathbb{R}_{\mathcal{G}^{*},\exp}\right)$
and $\phi\left(f\right)\in\mathbb{R}\left\llbracket \left(Y^{-1}\right)^{*}\right\rrbracket $
then $\phi\left(f\right)$ is generalized multisummable in the positive
real direction.
\item Let $\mathcal{B}\subseteq\mathrm{an}$ be a collection of restricted
analytic functions such that $\exp\restriction\left[0,1\right]\in\mathcal{B}$.
Let $f\in\mathcal{H}\left(\mathbb{R}_{\mathcal{B},\exp}\right)$ such
that $f\left(y^{-1}\right)$ is analytic at zero. Then $f\in\mathcal{H}\left(\mathbb{R}_{\mathcal{B}}\right)$. 
\end{enumerate}
\end{cor}
\begin{proof}
$\ $
\begin{enumerate}
\item By Binet's second formula, if $\Gamma\restriction\left(0,+\infty\right)$
were definable in $\mathbb{R}_{\mathrm{an}^{*},\exp}$, then so would
the function
\[
\log\left(\Gamma\left(\frac{1}{x}\right)\right)+\left(\frac{1}{x}-\frac{1}{2}\right)\log x+\frac{1}{x}-\frac{1}{2}\log\left(2\pi\right).
\]
The Taylor expansion at $0^{+}$ of the above function is the Stirling
series, which is divergent (see for example \cite[12.33]{whittaker-watson}).
But by Lemma \ref{lem: image of (gen) series}, $\mathbb{T}_{\mathrm{an}^{*}}\cap\mathbb{R}\left\llbracket \left(Y^{-1}\right)^{*}\right\rrbracket =\mathbb{R}\left\{ \left(Y^{-1}\right)^{*}\right\} $,
so the Stirling series is not in the image of $\phi$.
\item This is an immediate consequence of Lemma \ref{lem: image of (gen) series}(2).
\item It suffices to replace $\mathcal{B}$ by the GQC obtained from $\mathcal{B}$
by closing under the (clearly definable) operations listed in \cite[1.8 and 1.15]{rolin_servi:qeqa},
and to apply Lemma \ref{lem: image of (gen) series}(1). 
\end{enumerate}
\end{proof}
\begin{example}
In particular, Corollary \ref{cor: applications}(2) implies that
no solution $y\left(x\right)$ to Euler's differential equation $x^{2}y'=y-x$
for $x>0$ is definable in $\mathbb{R}_{\mathcal{G}^{*},\exp}$: if
this were the case, then since $\phi$ is truncation closed, there
would be one such solution whose transasymptotic expansion is exactly
Euler's series, which is not generalized multisummable in the positive
real direction (see for example \cite[Section B5]{Loray_series_divergentes}).
\end{example}
\bibliographystyle{alpha}

\begin{thebibliography}{DMM01}
	
	\bibitem[ADH17]{adh-bigbook}
	M.~Aschenbrenner, L.~{van den} Dries, and {J. van der} Hoeven.
	\newblock {\em Asymptotic differential algebra and model theory of
		transseries}, volume 195 of {\em Annals of Mathematics Studies}.
	\newblock Princeton University Press, Princeton, NJ, 2017.
	
	\bibitem[DD88]{vdd:d}
	J.~Denef and {L. van den} Dries.
	\newblock $p$-adic and real subanalytic sets.
	\newblock {\em Ann. of Math. (2)}, 128(1):79--138, 1988.
	
	\bibitem[DMM94]{dmm:exp}
	{L. van den} Dries, A.~Macintyre, and D.~Marker.
	\newblock The elementary theory of restricted analytic fields with
	exponentiation.
	\newblock {\em Ann. of Math. (2)}, 140(1):\ 183--205, 1994.
	
	\bibitem[DMM97]{dmm:series}
	{L. van den} Dries, A.~Macintyre, and D.~Marker.
	\newblock Logarithmic-exponential power series.
	\newblock {\em J. London Math. Soc. (2)}, 56(3):417--434, 1997.
	
	\bibitem[DMM01]{vdd_macintyre_marker:log-exp_series}
	{L. van den} Dries, A.~Macintyre, and D.~Marker.
	\newblock Logarithmic-exponential series.
	\newblock In {\em Proceedings of the {I}nternational {C}onference ``{A}nalyse
		\& {L}ogique'' ({M}ons, 1997)}, volume 111, pages 61--113, 2001.
	
	\bibitem[DS98]{vdd:speiss:gen}
	{L. van den} Dries and P.~Speissegger.
	\newblock The real field with convergent generalized power series.
	\newblock {\em Trans. Amer. Math. Soc.}, 350(11):4377--4421, 1998.
	
	\bibitem[DS00]{vdd:speiss:multisum}
	{L. van den} Dries and P.~Speissegger.
	\newblock The field of reals with multisummable series and the exponential
	function.
	\newblock {\em Proc. London Math. Soc. (3)}, 81(3):513--565, 2000.
	
	\bibitem[Hoe06]{vdh:transs_diff_alg}
	{J. van der} Hoeven.
	\newblock {\em Transseries and real differential algebra}, volume 1888 of {\em
		Lecture Notes in Mathematics}.
	\newblock Springer-Verlag, Berlin, 2006.
	
	\bibitem[KRS09]{krs}
	T.~Kaiser, J.-P. Rolin, and P.~Speissegger.
	\newblock Transition maps at non-resonant hyperbolic singularities are
	o-minimal.
	\newblock {\em J. Reine Angew. Math.}, 636:1--45, 2009.
	
	\bibitem[KS19]{kaiser_speissegger:analytic_continuation}
	T.~Kaiser and P.~Speissegger.
	\newblock Analytic continuations of {$\log$}-{$\exp$}-analytic germs.
	\newblock {\em Trans. Amer. Math. Soc.}, 371(7):5203--5246, 2019.
	
	\bibitem[Lor98]{Loray_series_divergentes}
	F.~Loray.
	\newblock {\em Analyse des s{\'e}ries divergentes}, volume Quelques aspects des
	math{\'e}matiques actuelles of {\em Math{\'e}matiques pour le 2e cycle}.
	\newblock Ellipses, 1998.
	
	\bibitem[Neu49]{neumann:ordered_division_rings}
	B.~H. Neumann.
	\newblock On ordered division rings.
	\newblock {\em Trans. Amer. Math. Soc.}, 66:202--252, 1949.
	
	\bibitem[RS15]{rolin_servi:qeqa}
	J.-P. Rolin and T.~Servi.
	\newblock Quantifier elimination and rectilinearization theorem for generalized
	quasianalytic algebras.
	\newblock {\em Proc. Lond. Math. Soc. (3)}, 110(5):1207--1247, 2015.
	
	\bibitem[RSS07]{rss}
	J.-P. Rolin, F.~Sanz, and R.~Sch{\"a}fke.
	\newblock Quasi-analytic solutions of analytic ordinary differential equations
	and o-minimal structures.
	\newblock {\em Proc. London Math. Soc.}, 95(2):413--442, 2007.
	
	\bibitem[RSS23]{rss:multisummability_generalized}
	J.-P. Rolin, T.~Servi, and P.~Speissegger.
	\newblock Multisummability for generalized power series.
	\newblock {\em Canad. J. Math.}, pages 1--37, 2023.
	
	\bibitem[RSW03]{rsw}
	{J.-P.} Rolin, P.~Speissegger, and {A. J.} Wilkie.
	\newblock Quasianalytic {D}enjoy-{C}arleman classes and o-minimality.
	\newblock {\em J. Amer. Math. Soc.}, 16(4):751--777, 2003.
	
	\bibitem[WW27]{whittaker-watson}
	E.~T. Whittaker and G.~N. Watson.
	\newblock {\em A course of modern analysis}.
	\newblock Cambridge University Press, Cambridge, 1927.
	
\end{thebibliography}
\def\cprime{$'$}

\end{document}